\documentclass[12pt]{article}

\usepackage{amsmath,amsfonts,mathtools}
\usepackage{amsthm}
\usepackage{amssymb}
\usepackage{amscd}
\usepackage{xypic}
\usepackage{verbatim}
\usepackage[plainpages=false,colorlinks,hyperindex,pdfpagemode=None,bookmarksopen,linkcolor=red,citecolor=blue,urlcolor=blue]{hyperref}
\usepackage{pdflscape}
\usepackage{stmaryrd}
\usepackage{mathabx}
\usepackage{multirow}
\usepackage{appendix}
\usepackage{color}

\swapnumbers
\theoremstyle{plain}

\renewcommand{\marginpar}[1] {  }
\renewcommand{\comment}[1] {  }

\usepackage[all]{xy}
\catcode`\é=\active\def é{\'e}
\catcode`\è=\active\def è{\`e}
\catcode`\à=\active\def à{\`a}
\catcode`\ù=\active\def ù{\`u}
\catcode`\ê=\active\def ê{\^{e}}
\catcode`\î=\active\def î{\^{i}}
\catcode`\ô=\active\def ô{\^{o}}
\catcode`\û=\active\def û{\^{u}}
\catcode`\ç=\active\def ç{\c{c}}

\newtheorem{theo}{Theorem}[section]
\newtheorem*{theo-nn}{Theorem}
\newtheorem{lem}[theo]{Lemma}
\newtheorem{prop}[theo]{Proposition}

\newtheorem{defi}[theo]{Definition}

\theoremstyle{remark}
\newtheorem{rem}[theo]{Remark}
\numberwithin{equation}{section}

\def \a{\mathfrak{a}}

\def\g{\mathfrak{g}}

\def\i{\mathfrak{i}}

\def\s{\mathfrak{s}}

\def\c{\mathfrak{c}}

\def\CC{\mathfrak{C}}

\def\ep{\varepsilon}

\def\vid{\emptyset}

\def\cA{{\mathcal A}}

\def\cC{{\mathcal C}}

\def\cE{{\mathcal E}}
\def\cF{{\mathcal F}}

\def\cH{{\mathcal H}}

\def\cL{{\mathcal L}}

\def\cP{{\mathcal P}}

\def\cS{{\mathcal S}}

\def\cW{{\mathcal W}}

\DeclareFontFamily{OT1}{rsfs}{}
\DeclareFontShape{OT1}{rsfs}{n}{it}{<-> rsfs10}{}
\DeclareMathAlphabet{\mathscr}{OT1}{rsfs}{n}{it}

\def\A{\mathbb A} 
\def\C{\mathbb C}

\def\Q{\mathbb Q}
\def\R{\mathbb R}

\def\N{\mathbb N}
\def\Q{\mathbb Q}
\def\R{\mathbb R}

\def\gg{ _{[G]}}
\def\ggp{ {{[G]_P}}}

\def\-{\setminus }

 \def\beq{\begin{equation}}
\def\eeq{\end{equation}}
\newenvironment{res}
               {\begin{equation}\begin{minipage}{0.85\textwidth}}
               {\end{minipage}\end{equation}}
\def\ber{\begin{res}}
\def\eer{\end{res}}
\newenvironment{res-nn}
               {$$\begin{minipage}{0.85\textwidth}}
               {\end{minipage}$$}
\def\qed{{\null\hfill\ \raise3pt\hbox{\framebox[0.1in]{}}\break\null}}

\newenvironment{dedication}
{           
\thispagestyle{empty}
\itshape             
\raggedleft          
}
{\par 
}

\author{Patrick {Delorme}}  
\title{On the Spectral  Theorem of Langlands}

\begin{document}

\maketitle  

\begin{dedication}
\`A Chantal
\end{dedication} 
\begin{abstract}
We show that the Hilbert subspace  of $L^2(G(F)\backslash G(\A))$ generated by wave packets of Eisenstein series built from discrete series is the whole space.
Together with the work of Lapid \cite{L1}, it achieves a proof of the spectral theorem of R.P. Langlands  based on the work of J. Bernstein and E. Lapid \cite{BL} on the meromorphic continuation of Eisenstein series from 
I use truncation on compact sets as J. Arthur did for the local trace formula in \cite{Alt}.
\end{abstract}

\section{Introduction}

We denote by $\overline{\cE}$ the isometry introduced by E. Lapid in \cite{L1}, Theorem 2, whose proof involves the meromorphic continuation of of Eisenstein series built 
from discrete data. 
A  dense subspace of its image is generated  by  wave packets of these Eisenstein series. We show that this  image is equal to $L^2(G(F)\backslash G(\A))$. 
This is what Lapid  calls the second half of the proof of the spectral theorem of Langlands.\\ 
It achieves a new proof of the spectral theorem of R.P. Langlands (\cite{Lan}, \cite{MW}). 
Notice that the proof of Langlands describes the spectrum as residues of Eisenstein series built from cuspidal data.\\ 
One uses the notion of temperedness of automorphic forms introduced  by J. Franke \cite{F}(cf. also \cite{W1} section 4.4). 
We show this notion of temperedness is weaker than the notion  of temperedness introduced by Joseph Bernstein in \cite{B}.\\
We prove that for bounded sets of  unitary parameters, the Eisenstein series are uniformly tempered when the parameter is unitary and bounded.  
One uses for this  that the growth of an automorphic form is controlled by the exponents of its constant terms and that the constant terms 
of Eisenstein series are explicited. \\
Some wave packets are in the Harish-Chandra Schwartz space (cf. \cite{L2}): 
this  is due to J. Franke \cite{F}, section 5.3,  Proposition 2 (2) but his proof  rests on the work of Langlands. 
We give here a selfcontained proof (cf. Proposition \ref{wps}) which uses  a uniform bound for exponential polynomials of one variable in terms of its exponents (cf. \cite{L1}) 
together with the general scheme of Harish-Chandra's study of wave packets in the Schwartz space in the real case (cf.\cite{HC},  see also \cite{BCD}).\\
Then, one shows that if  the image of $\overline{\cE}$ is not the full space, there would exist a tempered automorphic form orthogonal to these wave packets: the proof  is  similar to what we did  for real  symmetric spaces (cf. \cite{CD}, \cite{D2}). \\ 
One can compute an explicit  asymptotic formula for the truncated inner product of this form with an Eisenstein series. 
Actually, here,  we use ''true truncation'' i.e. truncation on compact sets as in \cite{Alt}. 
This uses partitions of $G(F)\backslash G(\A)$ depending on the truncation parameter (cf. \cite{AtrI}).\\
By a process of limit, as in \cite{D2}, one computes the scalar product of this form with a wave packet of Eisenstein series.  
Then one shows that it implies that this form is zero. A contradiction which shows that $\overline{\cE}$  is onto.
 
\section*{ Acknowledgment} I  thank more than warmly  Raphael Beuzart-Plessis  for his constant help. 
I thank also warmly Jean-Loup Waldspurger  for numerous comments  on a previous version of this work. 
I thank also warmly Pascale Harinck for her careful reading and suggestions to improve the writing.

\section{Notation} 
We introduce the notation  for functions $f$, $g$ defined on  a set $X$ with values in $\R^+$:
$$f (x) << g(x), x \in X$$ if there exists $C >0$ such that $f(x )<C g(x) , x\in X$.
We will denote this also $f<< g$.\\ 
Let us denote 
$$ f (x) \sim g(x), x \in X$$
if and only iff $ f<<g$ and $g<< f$.\\
If moreover $f$ and $g$ take values greater or equal to 1, we write:
$$f(x) \approx g(x)$$  if there exists $N>0$ such  that
$$g(x)^{1/N} << f(x) << g(x)^N, x \in X.$$
Let $F$ be a number field and $\A$ its adele ring. 
If  $G$ is   an algebraic group defined over $F$, we denote its unipotent radical by $N_G$ . 
Then the real vector space $\a_G$ is defined as usual, as well as the canonical morphism $H_G: G(\A)\to \a_G$. Let $G(\A)^1$ be its kernel.\\
From now on we assume that $G$ is  reductive. 
Let $P_0$ be a parabolic subgroup of $G$ defined over $F$ and minimal for this property. 
Let $M_0$  be a Levi subgroup of $P_0$.
We have the notion of standard  and  semi-standard parabolic subgroup  of $G$. 
Let $K$ be a good maximal compact subgroup of $G(\A) $ in good position with respect to $M_0$. \\
If $P$ is  a semistandard parabolic subgroup of $G$ we extend the map $H_P$ to a map
$$H_P: G(\A)\to \a_P$$
in such a way that $H_P(pk)= H_P(p)$ for $p\in P(\A), k \in K$.\\
We have a  Levi decomposition $P=M_PN_P$.
Let $A_P$  be the  maximal split torus of the center of  $M_P$ and $A_0=A_{M_0}$.\\
Let $G_\Q$ be the restriction of scalar  from $F$ to $\Q$  of $G$.  
We denote by  $A_P(\R)$  the group of real points of the maximal split torus  of the center of $M_{P, \Q}$ and $A_P^{\infty}$ the neutral component of this real Lie group. 
The map $H_P$ induces  an isomorphism between the neutral component $A_P^\infty$ to $\a_P$. 
The inverse map will be denoted $exp$ or $exp_{P}$.\\ 
We define :
$$[G]_P = M_P(F) N_P(\A) \backslash  G( \A), [G]=[G]_G.$$
The map $H_P$ goes down to a map $[G]_P \to \a_P$.\\
The inverse image of $0$ by this map is denoted $[G]_P^1$ and one has $[G]_P=  [G]_P^1A^\infty_P $.\\
If  $P\subset Q$ are semistandard parabolic subgroups of $G$,  we have  the usual decomposition 
$$ \a_P= \a^Q_P \oplus \a_Q. $$
This allows to view elements of $\a_Q^*$ as linear forms on  $\a_P$ which are zero on $\a_P^Q$.\\
The adjoint action of $M_P$ on the Lie algebra of $M_Q\cap N_P$ determines  $\rho_P^Q \in  \a_P^{Q,*} $. 
If $Q= G$ we omit $Q$ of the notation and we denote $\rho$ for $ \rho_{P_0}$.\\ 
If $P$ is a standard parabolic subgroup of $G$, let $\Delta_0^P $ be the set of simple roots of $A_0$ in $M_P\cap P_0$ and $\Delta_P\subset \a_P^*$  the set of restriction to $\a_P$ of the elements of $\Delta_{P_0} \- \Delta_0^P$. 
If $Q\subset P$, one defines also $\Delta_P^Q$ as the set  of restrictions to $\a_P$ of elements of $\Delta_{0}^Q \-\Delta_0^P$. 
One has  also the set of simple  coroots $\check{\Delta}_P^{Q}\subset \a_P^Q$. 
By duality we get simple weights ${\hat \Delta}^Q_P $ denoted  $\varpi_\alpha, \alpha \in \Delta^Q_P$. 
If $Q= G$ we omit $Q$ of the notation.\\
We denote by $\a_0^+$ the closed Weyl chamber and $\a_P^+ $ (resp. $\a_P^{++} $) the set of $X \in \a_P$  such that $\alpha(X) \geq 0$ (resp. $>0$ ) for $\alpha \in \Delta_P$.\\
If $P$ is a standard parabolic subgroup of $G$, we say that $\nu \in \a_{P, \C}^* $ is subunitary (resp. strictly subunitary) if $Re \nu (X) \leq 0, X \in \a_P^+ $ 
(resp. if $Re \nu = \sum_{ \alpha \in \Delta_P} x_\alpha  \alpha $ where  $x_\alpha< 0$ for all  $\alpha$). 
If  $\nu \in \a^*_P$,  it  is viewed  as a linear form on $\a_0$  which is zero on  $\a_0^P$. Then  $\nu\in  \a_{P, \C}^*  $ is subunitary if and only if  one has: 
\beq \label{subunitary} Re \nu (X) \leq 0, X \in \a_0^+ \eeq 
This follows from the fact $\alpha \in \Delta_P$ is proportional  to $\sum_{\beta \in \Sigma^+_0, \beta_{ \vert  \a_P} = \alpha}\beta$ where $\Sigma^+_0$ is the 
set of positive roots. 
In fact this sum is invariant by the Weyl group  generated by the reflections around roots which are $0$ on $\a_P $.\\
Let $W$ be the Weyl group of $(G,A_0)$. 
If $P, Q$ are standard parabolic subgroups of $G$,  let $W(Q\backslash G/P) $ be a set of representatives of $Q\backslash G/P$ in $W$ of minimal length. 
If $s \in  W(Q\backslash G/P) $  the subgroup $M_P \cap s^{-1} M_Qw$  is the Levi subgroup of a parabolic subgroup $P_s$ contained in $P$  and $M_Q \cap sM_Ps^{ -1}$ is the Levi subgroup of a parabolic subgroup  $Q_w$ contained in $Q$. 
Let $W(P, Q)$ be the set of $w \in W(Q\backslash G/P) $ such that $w M_P w^{-1} \subset M_Q$. 
Let $W(P\vert Q)$ be the set of $w \in W(Q\backslash G/P)$ such that $w(M_P)= M_Q$. 
Hence: 
$$W (P\vert Q)= W(P,Q) \cap W(Q, P)^{-1} .$$ 
By a Siegel domain $\s_P$ for $[G]_P$, we mean a subset of $G(\A)$ of the form: 
\beq\label{defsiegel} \s_P = \Omega_0\{ exp X\vert X \in \a_0 , \alpha (X+T)\geq 0, \alpha \in \Delta_0^{P}\}K \eeq 
where $\Omega_0$ is a compact of $P_0(\A)^1$, $T \in \a_0$, such that $G(\A)= M_P(F) N_P( \A) \s_P$. 
Let 
$$\a_0^{+,P} := \{ X \in \a_0\vert \alpha (X) \geq 0,  \alpha \in \Delta_0^P\}.$$ 
Let us show  \ber \label{siegel} Any Siegel set $\s_P $ is contained in $\Omega_{N_P}  \{exp X\vert X  \in \a_0^{+,P}\}  \Omega $  where  $\Omega$  is a  compact  subset of  $G(\A)$ and $\Omega_{N_P}$ is a compact subset of $N_P( \A)$. \eer  
There is a compact subset, in fact reduced to a single element, $exp T$, $\Omega_1\subset A_0^\infty $ such  that 
$$\{ exp X\vert X \in \a_0 , \alpha (X+T)\geq 0, \alpha \in \Delta_0^{P}\} \subset \{\exp X\vert  X \in \a_0^{+,P} \}\Omega_1. $$  
The compact set $\Omega_0$ is a subset of $ \Omega_{ N_P} \Omega_{ P_0 \cap M_P}$, where $ \Omega_{ N_P} $(resp. $ \Omega_{ P_0 \cap M_P}$) is a  compact   subset  of  $N_P( \A)$ (resp. $(P_0 \cap M_P) ( \A)$). 
Then the conjugate by $exp-X $, $X\in \a_0^{+,P} $ of $\Omega_{P_0 \cap M_P} $ remains  in a compact set, $ \Omega_2$,  when $X$ varies in $ \a_0^{+,P}$. 
The compact subset $\Omega= \Omega_2\Omega_1K$ satisfies the required property.\\ 
Let $P $ be a standard parabolic subgroup of $G$. 
If $f$ is a function on $G(\A) $ with values in $\R^+$, we denote $f_{[G]_P} $ the function on $[G]_P$ defined by 
\beq f_{[G]_P}(g)=  inf_{\gamma \in M_P(F) N_P( \A)} f (\gamma g), g \in G(\A).  \eeq 
We fix  a height  $\Vert. \Vert $ on $G ( \A)$ (cf.  \cite{MW}, I.2.2). 
From \cite{BP}, Proposition A.1.1 (viii), one has 
\beq \label{norm} \Vert g \Vert \sim \Vert g \Vert_{[G]} , g \in \s_G\eeq 
Let us define:
$$\sigma( g)= 1 + log (\Vert g \Vert), g \in G(\A).$$
We have $$\Vert mk\Vert_{[G]_P}\approx \Vert m\Vert_{[M]}.$$
From  this and (\ref{norm}), one deduces: 
\beq \label{sisih} \sigma_\ggp  (g) \sim  \sigma (g)\sim  1+ \Vert H_0(g) \Vert , g \in \s_P. \eeq  
Let $g \in \s_P$.   From  (\ref{siegel}), one can write it $ g=\omega_0 exp X \omega$ with $\omega_0 \in \Omega_{N_P}, \omega \in \Omega, X \in \a_0^{+,P}$. 
Then taking into account 
$$\Vert g \Vert \approx \Vert exp X\Vert $$ 
one has $$\sigma( g) \approx \sigma (expX)$$ and from (\ref{sisih}):
$$1+ \Vert H_0(g)\Vert  \sim 1+\Vert X \Vert, g \in  \s_P$$
From this it follows:  
\beq \label{ssg}   \sigma_{[G]_P} (g)\sim  1+ \Vert H_0(g) \Vert \sim \sigma( g) \sim  1+\Vert X\Vert  , g \in \s_P.\eeq
We normalize the measures as in \cite{MW}, I.1.13. \\
If $X$ is a topological space, let $C(X)$ be the space of complex valued continuous functions on $X$. 
Let $ \Omega$ be a compact subset of $G[\A]$ and $[\Omega ]$ its image in $[G]$. 
Then, as $\Omega$ can be covered by a finite number of open sets on which the projection to $[G]$ is injective, one has:
\beq\label{section} \int_{[\Omega ] }f(  x) d x   <<   \int_{\Omega } f(G(F) g) dg, f \in C( [\Omega]), f \geq 0.\eeq 
Let  $B$ be a symmetric bounded neighborhood of $1$ in $G(\A)$ (a ball) and let $\Xi^\ggp(x)= (vol_\ggp xB)^{-1/2} $. 
\ber \label{equb}The equivalence class of the function does not depend of the choice of $B$.\eer
We have (cf. \cite{BCZ} section 2.4):
\begin{lem} \label{xiap}  
$$\Xi^\ggp(g) \sim  e^{\rho(H_0(g))}, g \in \s_P.$$
\end{lem}
Let $G_\infty $ be  the product of $G(F_v) $ where $v$ describes the archimedean places and let $U( \g_ \infty) $ be the enveloping algebra of the Lie algebra $\g_\infty$ of this real Lie group. 
We have similar definition for subgroups of $G$ defined over $F$.\\  
One has the  Schwartz space $\cC([G])$, denoted $ \cS([G]) $ in \cite{L2},  corollary 2.6.  
From Lemma \ref{xiap} and (\ref{ssg}), it can be defined (see \cite{L2}, section 1.1) as the space of functions $\phi $ in $C^\infty ([G])$ such that for all $n \in \N$ and $u\in U(\g_\infty)$:
$$\vert (R_u\phi)(x) \vert  <<  \sigma_{[G]}^{-n}( x)    \Xi^{[G]}(x) , x \in [G].$$  

\section{Tempered automorphic forms}

\subsection{Definition of temperedness}

The space of automorphic forms  on $[G]=G(F)\backslash G(\A)$, $\cA(G)$  is defined as in \cite{MW}, I.2.17. 
In particular they are $K$-finite. \\  If  $\phi\in \cA(G)$, it  has uniform moderate growth on $G ( \A)$ (cf. lc.  end of I.2.17, Lemma I.2.17 and Lemma I.2.5).\\ 
It means that there exists $r >0$ such that for all $u \in U( \g_\infty)$: 
$$\vert R_u \phi ( g) \vert <<  \Vert  g \Vert^r_{[G]}, g \in G(\A) .$$ 
Let $\Omega$ be a compact subset of $G(\A)$.  
Then one sees easily  that this implies that for all $u \in U( \g_\infty)$, one has 
\beq  \label{unifm} \vert R_u R_\omega \phi(g)\vert<<   \Vert  g \Vert^r , g \in G(\A) , \omega\in\Omega.\eeq 
If $P$ is a standard parabolic subgroup of $G$, the space of automorphic forms on $[G]_P=N_P(\A)M(F)\backslash G(\A)$ denoted $A(N_P(\A) M(F) \backslash G(\A))$ in l.c.  will be denoted   $\cA_P(G)$.  
The constant term along $P$ (cf. \cite{MW}, I.2.6), $\phi_P$,  of an element $\phi$ of $ \cA(G)$ is an element of $\cA_P(G)$.  
Similarly if $Q$ is  a standard parabolic subgroup contained in $P$ and $\phi \in \cA_P(G)$, $\phi_Q$ is a well defined element of $\cA_Q(G)$. 
Let $ \cA^n_P(G) $  be the space  of elements  of $ \phi  \in   \cA_P(G)$ such that: 
$$\phi(  exp  X g) = e^{\rho_P(X) }\phi (g), g \in G(\A), X \in \a_P.$$ 
If $\phi \in \cA_P(G) $, and $\lambda \in \a_{P,\C}^*$ we define 
$$ \phi_\lambda  (g)= e^{ \lambda (H_P(g))} \phi (g), g\in  G( \A). $$ 
We view $S(\a_P^*)$ as the space of polynomial functions on $\a_P$  and $S(\a_P^*) \otimes \cA^n_P(G)$ as a space of functions on $G(\A)$ by setting 
$(p\otimes \phi)(g)=p (H_P(g)) \phi(g)$. 
If $\phi \in \cA_P(G)$,  one can write uniquely: 
\beq \label{1} \phi  (g)= \sum_{\lambda \in \cE_P( \phi) } e^{\lambda(H_P(g))} (\phi_{0,\lambda}) (g) , \eeq 
where $\cE_P ( \phi) \subset \a_{P, \C}^*$,  $ \phi_{0,\lambda}$ is a non zero element of $ S(\a_P^*) \otimes \cA^n_P(G) $. 
The set $ \cE_P( \phi)$ is called the set of exponents of $\phi$. 
We define also $\cA^2_P(G)$ as the subspace of  elements  of $\cA^n_P( G) $ such that: 
$$\Vert \phi \Vert^2_P = \int_K  \int_ {A^\infty_{P}M_P( F)  \backslash M_P(\A) } \vert \phi( mk) \vert^2e^{- 2 \rho_P(H_P(m))} dm dk. $$ 
We will need the following variant  of \cite{MW}, Lemma I.2.10:
\ber\label{mw}   
a) The statement  of \cite{MW},  Lemma I.2.10 remains true if one changes $P$ to a standard parabolic subgroup and, in the conclusion, one changes $\alpha$ to 
$\beta_P:= \inf_{\alpha \in \Delta_0 \setminus \Delta_0^P} \alpha$.\\ 
b) If one replaces $m_{P_0}^\lambda (g)$ in the hypothesis by $m^\lambda _{P_0} (g) (1+ \Vert log( m_{P_0} (g))\Vert  )^n$ for some $n$, one can replace 
$m_{P_0} (g)^{\lambda -t \alpha}$  by $m_{P_0} (g)^{\lambda -t \beta_P} (1 +\Vert log( m_{P_0} (g))\Vert  )^n$  in the conclusion.\\ 
c) One can also replace in the statement $\Vert a \Vert ^r $ by $ (1 +\Vert log a\Vert ) ^r$ for $a \in A_G $.  
\eer 
To prove a) one has simply to replace  $\alpha$ by $\beta_P $  in (1) of  the proof. \\ 
To prove b) one has to use  that $(1+ \Vert( log m_{P_0} (g))\Vert  )^n $ is $U_0( \A)$ invariant  after (6) in the proof.\\ 
(c) is obtained by replacing  $\Vert a \Vert ^r $ by $ (1 +\Vert log a\Vert )^r$ in the proof.

\begin{lem}\label{abc} 
Let $P$ be a standard parabolic subgroup of $G$. Let $d >0$.  
Let $\phi\in \cA_P(G) $. The following conditions are equivalent:\\
a)   $$ \vert \phi(x)\vert << \Xi^\ggp (x) \sigma_\ggp (x)^{d}, x \in \ggp.$$
b) For all Siegel domains $\s_P$, one has:
$$ \vert \phi( g  )\vert  << e^{ \rho(H_0(g))}  ( 1 + \Vert H_0(g)\Vert   )^{d},  g \in \s_P.$$
c) For  every compact subset  $\Omega $ of $ G(\A)$, one has:   
$$  \vert \phi( exp X  \omega   )\vert  <<  e^{\rho(X)} (1+ \Vert X\Vert )^d, \omega \in \Omega, X \in \a_0^{+,P} .$$
\end{lem} 

\begin{proof} 
a)  is equivalent to  b)  follows from  Lemma \ref{xiap} and  (\ref{ssg}).\\  
To prove  c) implies b),  we choose (cf. (\ref{siegel})), a compact subset  $\Omega$ of $ G(\A)$  and a compact subset, $\Omega_{N_P}$ of $N_P( \A)$ such that 
$ \s_P \subset \Omega_{N_P} A_0^{\infty,+,P} \Omega $, where $A_0^{ \infty, +,P} = \{ exp X\vert X \in \a_0^{+,P}\}.$ 
One has for $g \in \s_P$: 
\beq \label{xh} \Vert X-H_0(g)\Vert << 1, g= \omega_{N_P} exp X \omega, X \in \a_0^{+,P}, \omega_{N_P}  \in \Omega_{N_P},  \omega \in \Omega .\eeq  
Then c) implies b) follows. \\ 
Similarly   b) implies c).
\end{proof}


\begin{defi}\label{deftemp}   
Let us define the space of tempered automorphic forms on $[G]_P$, $\cA^{temp}_P(G)$,  as the space of automorphic forms satisfying, 
as well as its derivatives by elements of $U( \g_\infty) $, the   equivalent properties a), b), c) of the preceding Lemma  for some $d$. 
\end{defi}

\begin{rem}  
This notion was introduced by J. Franke in \cite{F} (cf. also \cite{W1}, section 4.4  where the space of tempered form is denote $\cA_{log}(G)$).  
\end{rem}

\begin{lem}\label{key} 
Let $d> 0$. if $ \phi\in \cA(G)$  is such  that all its derivatives by elements of $U( \g_\infty)$ are in $L^2([G], \sigma_{[G]}^{-d} dx)$ then $ \phi $ is tempered. 
\end{lem}

\begin{proof}
Let us use the notation of \cite{B}, Lemma-Definition 3.3.  From the proof of this Lemma, one can take $dm_X(x)=\nu(x) dx$ where  $dx$ is the 
$G(\A)$-invariant measure on $X= [G]$ and $\nu= (\Xi^{[G]})^2$,  as it follows  from the proof of l.c. Lemma-Definition  3.3 (ii).\\
Let $k\geq dim G$ and $f $ be a continuously $k$-times differentiable  function  on $G(\A)$. 
Fix a basis $d_1, \dots, d_r$ of the space  $U(\g_\infty)^k$  of elements of $U(\g_\infty) $ of degree $\leq k$ and define:
$$Q(f) = \sum_i \vert d_i f \vert ^2.$$ 
Let $J$ be a compact open subgroup of $G(\A_f)$. 
Let  $C([G])^{k,J} $ be the space of continuously $k$-times differentiable and fixed by $J$. 
One has from the Key Lemma of l.c., p. 686:
$$\vert f(x) \vert^2 << \int_{xB} Q(f) dm_X= \int_{xB}Q(f) \nu(y) dy, x \in [G], f \in C([G])^{k,J}$$
Let $w = \sigma_{[G]}^d $.
Then 
$$\vert f(x) \vert^2  <<\int_{xB} Q(f)\nu w  w^{-1} dy,x \in [G] .$$
We use now that $w$ is a weight, in the sense of \cite{B}, Definition 3.1, as well as $\nu$ to get: 
$$\vert f(x) \vert^2 << \nu(x) w(x) \int_{xB} Q(f)  w^{-1} dx \leq  \nu(x) w(x) \int_{[G]}Q(f)   w^{-1} dx,x \in [G] .$$ 
As our hypothesis implies that $\int_{[G]}Q(f)  w^{-1} dx <\infty$,  this finishes the proof of  the Lemma. 
\end{proof} 

\subsection{Characterization of temperedness  and definition of the weak constant term}

\begin{prop}\label{subtemp} 
Let $\phi \in  \cA (G)$. It is in $\cA^{temp}_G$ if and only if  the  exponents of its constant term $\phi_P$ along any standard parabolic subgroup $P$ of $G$ are subunitary.
\end{prop} 

\begin{proof} 
Let us prove Let us show that the condition is necessary. Let $ \phi \in  \cA^{temp}(G) $.  Let $X\in \a_G$. 
Then there exists $d $ such that for all $g \in [G]$: 
$$\vert \phi(gexp tX)\vert<< (1+t)^d, t>0.$$
It follows from \cite{CM}, Proposition A.2.1, that the exponents of $\phi$ restricted to $\a_G$ are unitary. 
Let $P$ be a standard parabolic subgroup of $G$. 
Let $X \in \a^{++}_P\cap  \subset \a_0^+$, where $ \a^{++}_P $ is the interior of $\a_P^+$. 
Let  $ g \in G(\A), t\in \R $. From  the property c) of the definition of temperedness, applied to $\Omega= g $, one gets: 
\beq \label{phil} \vert \phi (exp tXg)\vert  <<  e^{ \rho(tX) }( 1+ t)^d, t>0. \eeq 
Due to (\ref{unifm}), one can apply (\ref{mw}) to the right translate by $g $ of $\phi$. 
Write $X=X_G + X^G$ with $X_G \in \a_G, X^G \in \a^G$. 
Applying (\ref{mw}) b) and c) for the parameter $t$ from \cite{MW}, Lemma I.2.10 large and  with $g$ of l.c. equals to $exp tX^G\in \s_G\cap G(\A)^1$, $a$ equals to 
$exp tX_G$, one gets, for $k>0$,
$$\vert \phi (exp tXg) -\phi_P(expt Xg )\vert  <<e^{\rho(tX)} e^{-k t}, t >0. $$
Together with (\ref{phil}) this implies that the exponential polynomial in $t$, $ \phi_P(expt Xg )$ satisfies: 
$$ \vert \phi_P(expt Xg )\vert  << e^{\rho(tX)} ( 1+ t)^d, t >0.$$
There is a dense open set $O$  in $\a^{++} _P$ such that different exponents of $\phi_P$ take different values on any element of $O$. 
We use the notation of (\ref{1}).  
Then  \cite{CM}, Proposition A.2.1  gives:\\
If  $\phi_{P,0,  \lambda}  (g)\not=0 $:  
$$Re\lambda (X) \leq 0, \lambda \in \cE_P(\phi), X \in O.$$ 
As $\phi_{P,0,  \lambda} $ is not identically zero, one has $Re\lambda (X) \leq 0  $ for $ X \in O$, hence also for $X\in  \a^+_P$ by density. 
This achieves the proof of (i).\\ 
The sufficiency of the condition  follows from \cite{MW}, Lemma I.4.I. 
\end{proof}

\begin{defi}\label{defctw}
Let $\phi \in \cA^{temp}(G)$ and let $P$ be a standard parabolic subgroup of $G$. 
We define the weak constant term of $\phi$, denoted $\phi_P^w$ as the sum of the terms in (\ref{1}) corresponding to unitary exponents. 
It is an element of $\cA^{temp}_P( G) $ from the preceding Proposition applied to $M_P$ (see below Lemma \ref{transw} for a detailed proof). 
Let $\phi_P^-= \phi_P- \phi_P^w$  and let us denote $\cE_P^w( \phi) $ (resp. $\cE_P^-( \phi) $) the exponents of $\phi_P^w$ (resp. $\phi_P^-$). 
\end{defi}

\subsection{Transitivity of the weak constant term} \label{transi}
Let $ Q\subset P$ be standard parabolic subgroups of $G$. 
If  $ \phi$ is a function on  $G(\A)$ and $k \in K$, we define  a function on $M_Q(\A)$ by:
$$\phi^{ k, M_Q}(m_Q) = e^{-\rho_Q(H_Q( m_Q))} \phi (m_Qk), k \in K$$
where $\rho_Q \in \a_Q$ is the restriction of $ \rho$ to $\a_Q$, which can be extended to $\a_0$ by zero on $\a_0^Q$. 
One has the  following  immediate properties, by coming back to the definitions: 
\ber\label{trans}  
If $\phi\in \cA(G) $ one has:$$ \phi_Q = (\phi_P)_Q, (\phi_Q) ^{ M_Q,k} =[( (\phi_P)^{  M_P,k} )_{Q \cap M_P}]^{M_Q,1} , k \in K$$ 
\eer
Notice that the function in bracket is a function on  $M_P( \A)$, hence the upper index   $^{ M_Q,1} $ indicates that we multiply by $e^{- \rho_{Q \cap M_P} (H_Q (m))}$ the restriction of this function to $M_Q(\A)$. 
\begin{lem} \label{transw}
Let $Q \subset P$ be as above.\\
(i) If $\phi\in\cA_P^{temp} (G) $ the exponents  of $\phi_Q$  are subunitary and one can define $\phi_Q^w$ as  the sum of the terms  of $\phi_Q$ 
corresponding to unitary exponents.\\
If $\phi \in \cA^{temp}(G)$ one has:\\
(ii) 
$\phi_P^w$    is in $\cA_P^{temp}(G)$.\\
(iii)
$$\phi_Q^{w}= (\phi^{w}_P)^{w}_Q .$$
(iv) 
$$(\phi^w_Q) ^{M_Q,k} =[( (\phi^w_P)^{  M_P,k} )^w_{Q \cap M_P}]^{ M_Q,1} , k \in K$$
\end{lem}

\begin{proof}
(i) is proved  as Proposition \ref{subtemp} (i). \\
(ii) From (\ref{trans}) and Proposition \ref{subtemp} applied to $\phi$, one sees that the exponents of  $(\phi^w_P)^{k, M_P} $ are subunitary. 
Hence by this Proposition applied to $M_P$, one sees: $\phi_P^{w, M_P,k} \in  \cA^{temp}(M_P)$. 
Let $\Omega_{M_P} $ be a compact subset of $M_P(\A)$. Using $K$-finiteness, this gives that there exists $d \in \N$ such that: 
$$ \vert  \phi_P  ( exp X \omega k) \vert << e^{\rho(X)} ( 1+\Vert X \Vert )^d,  X \in \a_0^{+,P} , \omega \in \Omega _{M_P}, k \in K.$$
Every compact subset of $G(\A)$ is contained in a set of the form $N_P( \A) \Omega' K$ where $\Omega' $ is a compact subset of $M_P(\A)$. 
Hence, recalling the definition of $\cA^{temp}_P (G)$ (cf. Definition \ref{deftemp}), the preceding estimate achieves the proof of (ii).\\
(iii) Write $\phi_P=\phi_P^w + \phi_P^-$. Then none of the exponents  of $(\phi_P^-)_Q $  is  unitary. 
Hence   as $ \phi_Q = (\phi _P)_Q$  we get (iii).\\
(iv) follows from  the second assertion of (\ref{trans}).
\end{proof}

\subsection{A characterization of elements of square integrable automorphic forms}

Let us recall some facts from \cite{B}. 
With the notation there, one can take $dm_X(x)=(\Xi^{[G]})(x)^2dx$ where  $dx$ is the  $G(\A)$-invariant measure on $X= [G]$,  as it follows from the proof of l.c. Lemma 3.3(ii). 
Also $[G]$ is of polynomial growth of rank $d_G$ (cf. the definition of this notion in l.c. p.689) equal to the split rank of $G$ (cf. l.c.  Example 1, p.698). 
Then taking into account  what follows the definition  of polynomial growth in l.c. and the criterion p.685, one gets:
\beq\label{sixi} 
\int_{[G]}  (1+ \sigma_{[G]}(x) )^{-d_G}  \Xi(x)^2 dx <  \infty.
\eeq
It follows from  \cite{MW}, Lemmas I.4.1 and I.4.11 that:
$$\cA^2(G)\subset \cC(G).$$
From this and  the definition  of temperedness above, one has: 
\ber \label{polynomial}
For all $\phi \in \cA^{temp}(G)$ and $ \psi \in \cA^2(G)$, for all $X \in \a_G$, the integral 
$$\int_{ [G]^1 } \phi( g_1 exp X) \overline{\psi} ( ( g_1 exp X) dg_1$$ 
is absolutely convergent. 
It is denoted $(\phi, \psi)_G^X$.  
Moreover $X \mapsto (\phi, \psi)_G ^X$ is an exponential  polynomial in $X$.  
One defines similarly for $\phi \in \cA^2_P(G), \psi \in \cA^{temp}_P(G)$,  an exponential  polynomial, $p_P( \phi,\psi)$ on $\a_P$ by 
$p_P(\phi, \psi)(X)= (\phi  , \psi )^X_P $,  $X\in \a_P$, using integration on  $[G]_P^1$. 
\eer
We denote by $ \cA^{temp,c}(G)$ the  space of $ \psi\in \cA^{temp}(G)$ such that  for all $\phi \in \cA^2(G) $, the polynomial $p_P(\phi, \psi)$ is zero. 
We define similarly $\cA^{temp,c}_P(G)$. 
Then one has a direct sum: $\cA^{temp,c}_P(G) \oplus \cA^2_P(G)$ and one can define, for  $\phi= \phi_1+ \phi_2\in\cA^{temp,c}_P(G)\oplus\cA^{2}_P(G)$ and 
$\psi \in \cA^{temp}_P(G)$, an exponental  polynomial denoted $p_P(\phi,\psi)$ equal to $p_P(\phi_2, \psi)$.\\ 
With these definitions one has:

\begin{lem}\label{tempexp} 
(i) Let $Q$ be a standard parabolic subgroup of $G$ and  let  $\phi \in  \cA^{temp}_Q(G) \cap \cA_Q^n(G)$ such that $\phi_P^w= 0$ for any standard parabolic subgroup 
of $G$ with $P \subset Q$,  $Q\not=P$,  $P$ standard  then $\phi \in  \cA^{2}_Q(G)$.\\
(ii) If $ \phi \in \cA^{temp}(G) $ and  $\phi_Q^{w}\in  \cA_P^{temp, c}(G)$ for all  standard parabolic subgroup $Q$  of $G$, then $\phi=0$.
\end{lem} 

\begin{proof} 
(i)  We first prove the result for $Q=G$. 
Let us show that for any standard parabolic subgroup of $G$, $P\not=G$,  the exponents of $\phi_P$ are strictly subunitary. Let $\nu $ be such an exponent. 
From the hypothesis it is subunitary but  not unitary. If it is not strictly subunitary, there exists $ \alpha \in \Delta_P$ such that: 
$$Re\nu= \sum_{ \beta\in \Delta_P\setminus \{\alpha\} } x_{ \beta } \beta, x_\beta \leq 0.$$
Let $Q $ be the maximal parabolic subgroup  of $G$,  containing $P$, such that $\Delta_Q= \alpha_{ \vert \a_Q} $. 
Then $\nu_{ \vert \a_Q}$  is an exponent of $\phi_Q$. But it is clear that it is unitary, hence $\phi_Q^w$ is non zero which contradicts our hypothesis. 
Hence $\nu$ is strictly subunitary. 
Then (i) for $Q=G$ follows from Lemma I.4.11 of \cite{MW}.\\
Then the statement of (i) follows  from Lemma \ref{transw} (iv) and what we have just proved  for $M_Q$ instead of $G$.\\ 
Let us prove  (ii) by induction on the dimension on $\a_0^G$. If it is zero the claim is clear.  
Suppose now $dim \a_0^G >0$. By applying the induction hypothesis to $M_P$ for a strict standard parabolic subgroup $P$ of $G$ and Lemma \ref{transw}, 
one sees that $\phi_P^w=0$. 
Hence by (i), $\phi\in \cA^2(G)$.  As  $\phi$ is in $\in \cA^{temp,c} (G)$, one deduces from this  that $\phi=0$. 
\end{proof}

\section{Uniform temperedness of Eisenstein series}

\subsection{Exponents of Eisenstein series} \label{expo}

Let $P$ be a standard parabolic subgroup of $G$. Let $\phi \in \cA^2_P(G)$.  Let $E_P(., \phi, \lambda)$ be the Eisenstein series (cf. \cite{BL}, (1.1)). 
Let $P \subset Q$ be two standard  parabolic subgroups of $G$ and $w\in W(P\vert Q)$. 
One has the  operators $M(w, \lambda): \cA^2_P(G) \to \cA^2_Q(G)$ meromorphivc in $\lambda \in \a_{P, \C}^*$ ( cf. l.c. section 1.1, after (1.1).\\
One can define $E^Q_P (., \phi, \lambda)$ which is, when defined, in $\cA_Q (G)$ and is characterized by:
\beq \label{epq} 
E^Q_P (., \phi, \lambda)^{M_Q,k} = E_{P \cap M_Q}( .,  \phi^{ M_Q,k}, \lambda), k \in K.
\eeq 
They are analytic on the imaginary axis (cf. \cite{BL} Remark 1.3). 
We recall   the formula  for generic $\lambda $ (cf. \cite{L1}, Proposition 4, with the notation there)
\beq \label{eq} 
(E_P(., \phi, \lambda))_Q = \sum_{s \in W(Q\backslash  G /P)}  E _{Q_s}^Q (.,  M(s, \lambda) \phi_{P_s} ,s\lambda). 
\eeq 
The exponents of $E_P(., \phi, \lambda)_Q$ are given by \cite{L1}, equation  (13). 
Moreover they are subunitary by l.c. Lemma 6 for $\lambda $ unitary. 
Hence by Proposition \ref{subtemp}, $E_P(., \phi, \lambda)$  is  tempered for $\phi \in \cA^2_P(G), \lambda\in i\a_P^* $  and  the  weak constant term is given by: 
\beq \label{eqw} 
(E_P(., \phi, \lambda))^w_Q= \sum_{s \in W(P, Q)}  E _{Q_s}^Q (.,  M(s, \lambda) \phi, \lambda )
\eeq 
which is holomorphic in a neighborhood of $i \a_P^*$. 
The same is true for $E^-_Q(., \phi, \lambda)= E_Q(., \phi, \lambda)  -E_Q^w (., \phi, \lambda) $ whose exponents are contained in 
$$\cE_Q^-( \lambda)= \cup_{s\in W(Q\backslash \ G/P)\setminus W(P, Q) } \{ s ( \cE_{P_s}( \phi) +\lambda)_{\vert \a_Q}\}$$ 
Hence, by analyticity, this inclusion holds for all $\lambda$ in $i \a_P^*$. 
This implies:
\ber\label{real}
For  $\lambda \in i\a_P^*  $, $X\mapsto E_Q^-( exp X g), X \in \a_Q$ is an exponential polynomial with exponents in  a multiset  $\tilde{\cE}_Q ^{-} (\lambda) $, 
where  $\tilde{\cE}_Q ^{-}(\lambda)$  is built  from $\cE_Q^-( \lambda)$ by some repetitions, the multiplicities  depending on the multiplicities of the exponents of $ \phi$. 
Moreover the  real parts of the exponents above do not depend on $\lambda \in i\a_P^*$.
\eer 

\subsection{Uniform temperedness of Eisenstein series}

Let $\Lambda $ be a compact subset of $i\a_P^*$. 
\ber 
Let $\mu \in \a_0^{G, +} $, $n\in \N$.  
Let  $\cF_{\Lambda,  \mu,  n  }$ be the space of functions $F $ on $G(\A) \times \Lambda$ which satisfy  for every compact subset $\Omega $ of $G(\A)$ and 
$u\in U( \g_\infty)$: 
$$ \vert R_uF( exp X \omega, \lambda) \vert << (1+ \Vert X \Vert )^n e^{\mu (X )},  X \in  \a_0^+, \omega\in \Omega, \lambda \in \Lambda,$$ 
$$ F( exp X g, \lambda) =e^{\lambda(X)} F(g, \lambda), X \in \a_G, g \in G( \A), \lambda \in \Lambda. $$
\eer 

\begin{prop} \label{uni} 
Let $E: [G]\times \Lambda \to \C$ be defined by: $ E(g, \lambda)= E_P(g, \phi, \lambda)$. 
Then  there exists $n \in \N$ such that 
$$ E \in \cF_{\Lambda, \rho, n}$$
\end{prop}

\begin{proof}
We will need the fact that lemma 1.4.1 of \cite{MW} holds uniformly for a set of automorphic forms which is bounded in a space of functions with given moderate growth 
and whose constant terms uniformly satisfy the assumption of that lemma. This is easy to see from the proof given in  \cite{MW}. 
This applies to Eisenstein series from \cite{BL}, Corollary 6.5, from the holomorphy of Eisenstein series on the imaginary axis and from  (\ref{epq}), (\ref{eq}).  
\end{proof} 

Let us show:
\ber \label{etempc}
If $P \not =G$, $E_P(.,\phi, \lambda ) \in \cA^{temp,c}(G)$ for all $\phi \in \cA^2_P(G), \lambda \in i\a_P^*$.
\eer
Let $X\in \a_G$ and us look to 
$$I(\lambda)=  \int_{[G]^1}E(xexpX, \phi, \lambda) \overline{\psi}(xexpX)  dx, \lambda \in i\a_P^*$$
for $\psi \in \cA^2(G)$. 
From the uniform temperedness of Eisenstein series, it is a continous function in $\lambda$. 
Let $z$ be an element of the  center $Z([\g_\infty, \g_\infty])$ of the envelopping algebra of $[\g_\infty, \g_\infty ]$. 
Let us assume that $z^*$ is in the cofinite dimensional ideal of $Z([\g_\infty, \g_\infty] ) $ which annihilates $\psi$. 
One can assume that $\phi$ is $Z(\g_\infty)$ eigen  and let $p_z$ be the polynomial on $i\a_P^*$ such that 
$R_z E_P(., \phi, \lambda) = p_z( \lambda) E_P(., \phi, \lambda), \lambda \in i\a_P^*$. Then on one hand:
$$\int_{[G]^1}R_z E(xexp X, \phi, \lambda) \overline{\psi}(xexp X)  dx, \lambda \in i\a_P^*= p_z( \lambda)I(\lambda)$$ 
and on the other hand, taking adjoint,  this integral is zero.\\
Moreover by the cofinite dimension of the annihilator in $Z([\g_\infty, \g_\infty] )$  of $\psi$, there exists $z$ as above  such that $p_z$ in non identically zero. 
Then it follows, by continuity and density,  that $I(\lambda)$ is identically zero. This proves our claim.

\section{Wave packets}

\subsection{Difference of a tempered automorphic form with its weak constant term}

Let $Q$ be a parabolic subgroup of $G$. 
For $\delta>0$,  we define: 
$$ \a^{G,+}_{Q, \delta}= \{X \in \a_Q^{G,+}\vert  \alpha(X) \geq  \delta \Vert X \Vert , \alpha \in \Delta_Q\}.$$

\begin{lem}\label{diff1} 
Let $\phi \in \cA^{temp}(G)$. Let $\Omega $ be a compact subset of $G$. 
Let $\delta>0$. Then there exists $ \ep>0$  such that: 
$$\vert (\phi- \phi^w_Q )\vert  ( n_Q exp X  exp Y  \omega)  <<   (1+ \Vert X \Vert )^d e^{\rho(X)}  e^{\rho (Y) -\ep\Vert Y \Vert } ,$$ 
$$ n_Q \in N_Q( \A),  X \in \a_0^+, Y \in \a^{G,+}_{Q, \delta}, \omega \in \Omega.$$
\end{lem} 

\begin{proof}
Using that  conjugation by $exp-X$ and $exp-Y$ contracts $N_Q(\A)$ and as $N_Q (F)\backslash N_Q (\A)$ is compact, possibly changing $\Omega$, one is  reduced to prove a similar claim, but without $n_Q$. \\
Let $S_{Q,\delta} $ be the intersection  of  the unit sphere of $\a_Q$  with $ \a^{G,+}_{Q, \delta}$. It is compact. 
Let us look  to the family of exponential polynomials in $t\in \R$: 
$$p_{Y,X, \omega}(t):=\phi_Q( exp X exptY \omega)-\phi_Q^w (exp X expt Y\omega ), X \in \a_0^+, Y \in S_{Q, \delta}, \omega \in \Omega.$$
On one hand,  from the definition of temperedness of $\phi$ and (\ref{mw}) b) and c),  one gets that there exists $d \in \N$ such  that: 
$$\vert \phi_Q( expX exp tY\omega)\vert << (1 + \Vert X\Vert )^d  e^{ \rho(X)} (1+t)^d e^{t\rho(Y)} , X \in \a_0^+, Y \in S_{Q, \delta}, \omega \in \Omega, t >0.$$
On the  other hand  the temperedness of $\phi_Q^w$ (cf. Lemma \ref{transw}) and  the definition of the temperedness of $\phi_Q^w$ implies a similar bound for 
$\phi_Q^w (exp X expt Y\omega )$. 
Hence by difference it follows  that there exists $d \in \N$ such that: 
\beq \label{pyx} 
\vert p_{Y,X, \omega}(t)\vert   << (1 + \Vert X\Vert )^d  e^{ \rho(X)} (1+t)^d e^{t\rho(Y)} , X \in \a_0^+, Y \in S_{Q, \delta},  \omega \in \Omega, t >0. 
\eeq 
Moreover the exponents of these exponential polynomials  are  equal to $\mu(Y) + \rho(Y)$ where $ \mu$ is an exponent of $\phi_Q$  which is not imaginary. 
Hence its real part is equal to $\sum_{\alpha \in \Delta_Q} c_\alpha \alpha (X) $ with $Re c_\alpha\leq  0$, with at least one $Re c_\alpha$ non zero. 
Hence there exists $\ep'>0$  such that for $Y \in S_{Q, \delta}$,  $\mu(Y) < -\ep'$. \\
By applying  (\ref{pyx}) to $\Omega'$ such that $\Omega'$ contains $ \{exp tY \vert  \vert t \vert <\ep'', Y \in S_{Q, \delta}\} \Omega$, 
one gets  that the modulus  of these polynomials restricted to the interval $[-\ep'',  \ep'']$  is bounded by a constant times $(1 + \Vert X\Vert )^d  e^{ \rho(X)}$. 
Applying  Lemma 3 of \cite{L1} to the polynomials $ [(1 + \Vert X\Vert )^d  e^{ \rho(X)}]^{-1} p_{Y, X , \omega}$, one gets the required estimate. 
\end{proof}

\begin{lem}\label{diff}
Let $\phi \in \cA^{temp}(G) $. 
Let $\delta >0$ and  $\a^{G+}_{0, Q, \delta}:=\{ X \in \a^{G+}_0\vert \alpha(X) \geq \delta \Vert X\Vert, \alpha \in \Delta_{P_0} \setminus  \Delta_{P_0}^Q\}$. 
Let $\Omega$ be a compact subset of $G(\A) $.\\  Let $\phi \in \cA^{temp}(G)$.  
There exists $\ep>0$ such that   
$$ \vert \phi (n_Q exp X \omega) -\phi_Q^w  (n_Q exp   X \omega)\vert <<  e^{\rho(X) - \ep \Vert X \Vert }, n_Q \in N_Q( \A) , X \in \a^{G+} _{0, Q, \delta} , \omega \in \Omega.$$
\end{lem}

\begin{proof}
For  $X \in  \a^{G+}_{0, Q, \delta} $, let  $Y$ be the element of $ \a^G_Q$ such that $ \alpha ( Y)= \alpha(X), \alpha \in \Delta_{P_0} \setminus  \Delta_{P_0}^Q$.  
Then, looking to coordinates in $\a_Q^G$, one sees that  there exists  $\delta_1 >0$ such that $ \delta_1 \Vert Y \Vert \leq  \delta\Vert X \Vert  $.  
Hence $Y \in \a_{Q, \delta_1}^+ $. Moreover  as  $X \in \a^{G+}_{0, Q, \delta}$, $X'= Y-X $ is in $\a_0^{G+} $. 
One gets the required estimate by using the preceding lemma with $X'$ instead of $X$ as: 
$$ \Vert X' \Vert  << \Vert X \Vert + \Vert Y\Vert <<  \Vert X \Vert$$
and  if $ \alpha \in \Delta_{P_0} \setminus \Delta_{P_0}^Q$,  
$$ \delta \Vert X\Vert \leq  \alpha (Y)<< \Vert Y\Vert .$$
\end{proof}

\begin{lem}\label{diffes}
Let  $\Lambda$ be a bounded subset of $i\a^*_P$   and $\phi \in \cA^2_P(G)$. There exists $\ep>0$ such that: 
$$\vert E_P( n_Q  expX \omega, \phi, \lambda) -E_P( n_Q  expX \omega, \phi, \lambda)_Q^w \vert <<  e^{\rho(X) - \ep \Vert X \Vert },$$ 
$$ n_Q \in N_Q( \A) , X \in \a^{G+} _{0, Q, \delta} , \omega \in \Omega, \lambda \in \Lambda.$$
\end{lem} 

\begin{proof}
The proof is similar to the proof of the preceding lemma. 
One has  to prove an analogous of Lemma \ref{diff1} for Eisenstein series by  using Proposition \ref{uni}, 
that the real part of the exponents of $E_P(., \phi, \lambda)_Q  $ do not depend  of $\lambda \in i \a_P^*$ 
and the expression of the weak constant term of Eisenstein series (cf. (\ref{eqw})).
\end{proof}

\subsection{Wave packets in the Schwartz space}

\begin{prop} \label{wps}
Let $a $ be   a smooth compactly supported function on $i\a_P^*$ and $\phi\in \cA_P^2(G)$. Then the  wave packet 
$$E_a:= \int_{i\a_P^*}a(\lambda)  E_P (., \phi, \lambda)d\lambda$$ 
is in the Schwartz space $\cC([G])$. 
\end{prop}

\begin{rem} 
As already said in the introduction, this is due to Franke, \cite{F}, section 5.3, Proposition 2 (2). 
His proof rests on the main result of \cite{Ainner} for which Lapid in \cite{L1} has given a proof independent of \cite{Lan}. 
We give below a more selfcontained proof.
\end{rem}

\begin{proof} 
We proceed by  induction on the dimension of $\a_0^G$. 
The case  where $dim{\a_0}^G=0$  is immediate by  classical Fourier analysis on $\a_G$: 
the classical Fourier transform of a compactly supported function on $\R^n$ is in the Schwartz space. \\
Now we assume $dim\a_0^G >0$.  Let $S^+$ be the intersection of the unit sphere of $\a^G_0$ with $\a^{G,+}_0$.  
Let  $X_0$ in $S^+$. Let $Q$ be the standard parabolic subgroup of $G$ such that $X_0 \in \a_Q^{++}$. 
As $X_0\in S^+$, $Q$ is not equal to $G$. Let $\beta_Q(X):= \inf_{   \alpha \in  \Delta_{P_0} \setminus \Delta_{P_0}^Q} \alpha(X), X \in \a_0$. 
Then $\beta_Q(X_0)>0$.  We  choose   a neighborhood $S_0$  of $X_0 $ in $S^+$ such that 
$$\beta_Q( X)\geq \beta_Q(X_0)/2, X \in S_0.$$ 
Let  $\delta=  \beta_{Q}(X_0)/2$ . Then 
$$S_0 \subset \a_{0, Q, \delta}^{G, +} .$$
Let $\Lambda $ be the support  of  $a $. We use the notation of Proposition \ref{uni}. 
Let  $ E(.,\lambda):= E_P (., \phi, \lambda) $. Then $E(., \lambda)$ is the sum of 2 terms: $ E(.,\lambda)-E(., \lambda)^w _{  Q} $ and $ E(.,\lambda)_{Q}^w.$ 
Let  us  show that, for all $k \in \N$ one has: 
\ber \label{F} 
$\displaystyle{\vert \int_{ i\a_P^*} a(\lambda)  F( exp X_G exp tX \omega, \lambda)  d\lambda \vert << (1+\Vert X_G\Vert )^{-k} (1+t)^{-k} e^{t \rho(X)},}$\\ 
$t >0, X_G \in \a_G, X \in S_0, \omega \in \Omega, \lambda \in  \Lambda,$ 
\eer
when $F$ is any  of these two families of functions. 
The case where $F = E_{Q}^w$  follows from the induction hypothesis, using the formula for the weak constant  term of Eisenstein series (cf. (\ref{eqw}),  (\ref{epq})) 
and the fact that  in this formula $M(w, \lambda) $ is analytic in $\lambda$ (cf. \cite{BL}, Remark 1.3). 
Let us treat the case where $F= E-E^w_Q$. 
One knows from Lemma \ref{diffes} that there exists $\ep>0$ such that: 
$$\vert E_P(exp X_G  exptX \omega, \phi, \lambda) -E_P(expX_G   exptX \omega, \phi, \lambda)_Q^w \vert <<  e^{t\rho(X) - \ep t },$$ 
$$X_G \in \a_G,  X \in S_0 , \omega \in \Omega, \lambda \in \Lambda.$$ 
By multiplying by $a$ and integrating on $i\a_P$, we get (\ref{F}) for  $F=E-E_Q^w$ and $k=0$.  
One applies this to successive partial derivatives  of $a$ with respect to elements of  $\a_G$. 
Then using that $E_Q^w$ transforms under  $\a_G$ by $\lambda$ and applying integration by part  one gets the result for all $k$. 
One can do the same for $R_uE$,  $u \in U( \g_\infty)$.\\ 
As a finite number of $S_0$ covers $S^+$ this achieves to prove the Proposition.
\end{proof}

\section{An isometry }
We recall the statement of Theorem 2 of \cite{L1}.\\ 
Let $\cP_{st}$ be the set of standard parabolic subgroups of $G$. Let $P$ be a standard parabolic subgroup of $G$. 
Let $\cW_P$  be the space of compactly supported smooth functions on $i \a_P^*$ taking values in a finite dimensional subspace of  $\cA^2_P$. 
Write:
\beq \label{3} 
\Vert \phi \Vert^2_* = \int_ {i \a^*_P } \Vert \phi(\lambda) \Vert_P^2 d\lambda.
\eeq 
For $ \phi \in \cW_P$, let 
$$\Theta_{P,  \phi}(g)  = \int_{i\a_P^*} E_P( g, \phi(\lambda), \lambda) d\lambda.$$
Let $L^2_{disc} (A_M^\infty  M(F) \backslash M( \A))$ be the Hilbert sum of irreducible $M(\A)$-subrepresentations of $L^2 (A_M^\infty  M(F) \backslash M( \A))$.\\ 
If $P$ is a standard parabolic subgroup of $G$, let $\vert \cP(M_P)\vert$ be equal to the number of  parabolic subgroups having $M_P$ as Levi subgroup. 
Consider the space $\cL$  consisting of families of functions $F_P: i\a^*_P \to  Ind_{P(\A)}^{G( \A)}L^2_{disc} (A_M^\infty  M(F) \backslash M( \A))$ 
where $P$ describes the set of standard parabolic subgroups of $G$ such that: 
$$ \Vert (F_P)\Vert^2 = \sum_{P \in \cP_{st}}\vert \cP(M_P)\vert^{-1}   \Vert F_P\Vert_*^2 <\infty $$ 
and 
\beq \label{fqm} 
F_Q (w\lambda)= M(w, \lambda) F_P( \lambda), w\in W(P\vert Q), \lambda \in i \a_P^*. 
\eeq 
Let  $\cL'$ be the subspace of $\cL$ consisting of those families such that $F_P \in \cW_P$ for all $P$. 

\begin{theo} \label{lapid}  \cite{L1}, Theorem 2\\
The map  $\cE$ from $\cL'$ to $L^2( G(F) \backslash G( \A))$
$$(F_P)  \mapsto \sum_{P\in \cP_{st}}\vert \cP(M_P)\vert ^{ -1} \Theta_{P, F_P}$$ 
extends to an isometry $\overline{\cE}$ from $\cL $ to  $L^2( G(F) \backslash G( \A))$.
\end{theo}

\begin{lem} 
We take the notation of Proposition \ref{wps}.  In particular $P$ is fixed. Then $E_ a$ is in the image of  $\cE$. 
\end{lem} 

\begin{proof}  
For this  one has to define a family in $ \cL'$  whose image by $\cE$ is a non zero multiple of $E_a$. 
If  $\psi  $ is a map on $i\a_P^*$  with values  in $\cA_P^2(G)$ and $s\in W(P,P)$,  we define 
$$s. \psi ( \lambda):= M(s,s^{-1}  \lambda) \psi (s^{-1} \lambda).$$ 
From the properties of composition of the intertwining operators, this defines an action of the group $W(P\vert P)$.
Let 
$$F_Q( \lambda) = \sum_{s \in W(P\vert Q)}  M(s,s^{-1}  \lambda) \psi (s^{-1} \lambda).$$
It is an easy consequence of the product formula for intertwining operators (cf. \cite{BL}, Theorem  1.3 (4))  that  $F_P$ is invariant by the action of $W(P,P)$ and 
also that the family $(F_Q)$ satisfies (\ref{fqm}). Moreover  it is in $\cL'$ as the intertwining operators are analytic on the imaginary axis (cf. l.c. Remark 1.3). 
Then the functional equation for Eisenstein series  (cf. lc. Theorem 1.3 (3)), implies that the image of $(F_Q)$ by $\cE$ is a non zero multiple of $E_a$. 
\end{proof} 

\section{Truncated inner product}

If $Q$ is   a semistandard  parabolic subgroup  of $G$, let: 
$$\theta_Q (\lambda) =\prod_{\alpha \in \Delta_Q} \lambda( \check{\alpha}), \lambda \in \a_{Q, \C}^*. $$ 
Let $L_Q$ be the cocompact lattice  of $\a_G^Q$  generated by $\check{ \Delta}_Q$  and let $C_Q = vol (\a_Q^G /L_Q)$.\\ 
We fix a Siegel domain as in (\ref{defsiegel}) associated to a compact set $\Omega_0P_0^1(\A) $ and to $T_0\in \a_0$. 
We can choose $\Omega_0= \Omega_{N_0} \Omega_{M_0^1}$ where  $\Omega_{N_0}$ (resp. $\Omega_{M_0^1}$) is a compact subset of $N_0( \A)$ 
(resp.$M_0( \A)^1$) such that 
\beq \label{nnm} 
N_0(\A)= N_0(F) \Omega_{N_0} ,  M_0( \A)^1 =M_0(F) \Omega_{M_0^1}.
\eeq
If $C$ is a subset of $\a_0$ we define $M_0(C)= \{m\in M_0(\A) \cap G(\A)^1 \vert H_0(m) \in C\}$ which is right invariant by $M_0(\A)^1$. 
We take $T $ dominant and regular in $\a_0^G$. 
We let $d_{P_0}(T)= \inf_{\alpha \in \Delta_{P_0} } \alpha(T)$ and if $Q $ is a standard parabolic subgroup of $G$, $T_Q$ is the orthogonal projection of $T$ on $\a_Q$.\\ 
If $Q$  is a standard parabolic subgroup of $G$, we define  the convex  set $C_T^Q $   of $\a_0^G$ by 
$$C_T^Q  =\{X \in \a_0^G \vert  \alpha (X-T_0) \geq  0, \varpi_\alpha ( X-T) \leq   0, \alpha \in \Delta_{0}^Q, \beta (X-T)>0, \beta \in \Delta_Q \}.$$ 
Notice that $C_T^G$ is compact. \\ 
Let $T_{M_Q}= T-T_Q, T_{0, M_Q}= T_0-T_{0, G}$. 
Let us define $C_{T_{M_Q}}^{M_Q}\subset \a_0^{M_Q} \subset \a_0^G$  is defined with  $T_{M_Q} = T-T_Q $ instead of $T$ and $T_{0, M_Q} $ instead of $T_0$. 
Let $ \a^{G,++}_Q(T)= T_Q + \a^{G,++}_Q$. 
We have:  
\beq \label{cqt=} 
C_T^Q = C_{T_{M_Q}}{M_Q}  + \a^{G,++}_Q (T)
\eeq

We define 
$$\CC^G_T= G(F) \Omega_{N_0} M_0 ( C_{T}^G )  K\subset [G]$$  
which is compact. Using (\ref{nnm}), one has: 
$$\CC^G_T= G(F) N_0( \A) M_0 ( C_{T}^G )  K. $$
Replacing $ N_0$ by $N_0 \cap M_Q$ and $G$ by $M_Q$  we define $\CC_T^{M_Q}\subset [M_Q]$ by: 
\beq \label{ccm} 
\CC_{T_{M_Q}}^{M_Q} =M_Q(F) (N_0\cap M_Q)(\A)M_0 ( C_T^{M_Q} )(K\cap M_Q(\A)).
\eeq
which is independent of the choice of $\Omega_0$. \\ 
We define 
\beq  \label{ccq} 
\CC_T^Q  =  Q(F) N_0( \A)  M_0 (C_T^Q )  K\subset  Q(F) \backslash G(\A)^1.
\eeq 
Then  $ \CC_T^Q  $ is $N_Q(\A)$ invariant as 
$$N_Q(\A)Q(F)N_0(\A)= Q(F)N_Q( \A) N_0( \A)= Q(F) N_0( \A). $$ 
As $N_Q( \A)N_0( \A) = N_Q( \A)(N_0 \cap M_Q)(\A)$ one has from (\ref{cqt=}): 
\beq \label{cctq=}
\CC_T^Q  =  N_Q(\A) exp (\a^{G,++}_Q (T))  \CC^{M_Q}_{T_{M_Q}}  K\subset  Q(F) \backslash G(\A)^1. 
\eeq 
We say that a strictly $P_0$-dominant $T\in \a_0^G$ is sufficiently regular if there exists a sufficiently large $d>0$ with $d_{P_0}(T)\geq d$. 
We have the following result due to Arthur (\cite{AtrI}, Lemma 6.4).
\ber\label{partition}
Let $T$ be sufficiently regular. \\ 
(i) For each standard parabolic subgroup $Q$ of $G$, viewing $\CC_T^Q$ as a subset of  $Q(F) \backslash G( \A)$, the  projection to $[G]$ is injective on this set. 
Its image is still  denoted $\CC_T^Q $.\\ 
(ii)  The $\CC_{T}^{Q} $ form a partition of $[G]^1$.
\eer
For a compactly supported function $f $ on $\CC_T^Q$  we have, using (\ref{cctq=}): 
\beq \label{intf}  
\int_{\CC_{T}^Q}  f( x) dx 
= \int_{(N_Q(F)  \backslash N_Q( \A ))\times {  \a^{G,++}_Q(T)} \times { \CC_{T_{M_Q}}^{M_Q}} \times K } f(n_Q exp X m^1_Q k)e^{- 2 \rho_Q(X)}  dn_Q dX dm^1_Q dk. 
\eeq 
as  follows  from the   integration formula  on $G(\A)$  related to the decomposition $ G(\A)^1 = N_Q(\A)exp \a^G_Q  M_Q(\A)^1 K$. 
Here  $dm^1_Q$ is the measure on $[M_Q]^1$.\\
Let  
$$\a^G_{0, -} =  \{ X \in \a_0^G  \vert \omega( X) \leq  0, \omega \in \hat{\Delta}_0 \}$$ 
be  the cone generated  by the negative  coroots   and  
$$\a^G_{Q,- }(T) =\{ X \in \a_Q^G  \vert \omega( X-T) \leq  0, \omega \in \hat{\Delta}_Q \} = T_Q  + \a_{Q} \cap \a_{0,-}^G.$$ 
Let $ p$ be an exponential polynomial with unitary exponents on $ \a_Q$ and $Z \in \a_G$. If $\mu\in \a_{Q, \C} ^{*}$ has its real part strictly $Q$-dominant, the integral: 
$$\int_ { \a^G_{Q,- }(T) }  e^{\mu(X+Z)}  p(X)   dX$$ 
is convergent and has a meromorphic continuation in $\mu$. When it is defined, its value in $\lambda \in i\a_Q^*$ is denoted:
$$\int^*_ {Z + \a^G_{Q,- }(T) }  e^{\lambda(X)}  p(X)  dX.$$ 
We use the notation following  \ref{polynomial}. 
We define  for $ \phi\in \cA^2_Q(G)\oplus \cA^{temp,c}_Q(G)$, $\lambda \in \a_Q^* $  and $ \Psi \in \cA^{temp}_Q (G)$,  $Z \in \a_G$:
\beq \label{rtp}
r^T_Q ( \phi_\lambda , \Psi )^Z  = \int^*_ {Z + \a^G_{Q,- }(T) }  e^{\lambda(X)} p_Q (\phi, \Psi)(X)  dX. 
\eeq 
Let $Z\in \a_G$. If  $p$ is a polynomial on $\a_Q$, we define $p^Z$  the exponential polynomial on $\a_Q^G$ defined by: 
$$p^Z (X)= p(X+Z), X \in \a_Q^G$$ 
and $p^Z( \partial)$ its Fourier transform viewed as a differential operator on $\a_Q^{G,*}$. \\ 
Recall that $C_Q$ has been defined in the beginning of this  section. One has:
\beq \label{rtp=} 
r^T_Q ( \phi_\lambda , \Psi )^Z  
= C_Q   \sum_{ \mu\in \cE_Q(\Psi) }e^{(\lambda-\mu)(T_Q+Z)}[p(\phi, \Psi_{0, \mu})^{Z} ( \partial)\theta_{ Q} ]( \lambda-\mu), 
\eeq 
One can define $r_T  ( \psi, \Psi)$ where $ \psi $ is a linear combination  of $ \phi_\lambda$. 
If $ \Phi $ is a function on $G(\A)$ and $Z \in \a_G$,  one defines a function on $G( \A)^1$ by: 
$$\Phi^Z ( g^1) = \Phi (  g^1exp Z), g^1 \in G( \A)^1 $$ 

\begin{theo} \label{tr} 
Let $\Phi$ be an element of  $\cA^{temp}(G)$ and $\phi  \in \cA^2_P(G)$. 
We denote by $E(., \lambda)$ the function $E_P (. , \phi, \lambda) $. 
Let: 
$$\Omega_{P_0}^T (E (\lambda),  \Phi)^Z:=\int_{  \CC_T^G} E(x,  \lambda)^Z   \overline{\Phi}^Z (x) dx,$$ 
$$ \omega_{P_0}^T( E(\lambda),\Phi )^Z:=\sum_{Q \in \cP_{st} } r^T_Q  (E(\lambda)_Q^w  , \Phi^w_Q)^Z .$$ 
(i) Let $ \cH^c$  be the subset of $\lambda \in i\a^*_P$ where  the summands of $\omega^T_{P_0}(E(\lambda, \Phi)$ are analytic for all $Z$. 
From (\ref{rtp=}), this set contains the complementary set of a finite union of hyperplanes. 
The function $\omega_{P_0} ^T(E( \lambda), \Phi)^Z$  on $\cH^c$ extends to an analytic  function on $i \a_P^*$ denoted in the same way.\\ 
(ii)
Let $\delta>0$. Let $\Lambda$ be a bounded set of $i\a_P^*$. 
There exists $k \in \N$, $\ep>0$ such that  the difference 
$$ \Delta_{P_0} ^T( E(\lambda) , \Phi)^Z:= \Omega_{P_0}^T (E (\lambda),  \Phi)^Z- \omega_{P_0}^T( E(\lambda),\Phi )^Z$$ 
is an $O( e^{-\ep \Vert T \Vert }( 1+ \Vert Z\Vert^k)$  for $\lambda \in \Lambda$, for  $T$ such that $d_{P_0} (T) \geq \delta \Vert T \Vert $, $Z \in \a_G$.
\end{theo} 

The proof is by induction on $dim \a_0^G$. The statement is clear for $dim \a_0^G=0$. We suppose that the Theorem is true for all groups $G'$ with $dim \a_0^{G'}< dim \a_0^G$. 


\begin{lem}
Let $ k_0, \delta> 0$. Then if  $\Lambda$ is  a bounded subset  of $\cH^c$, there exists $C>0$, $k\in \N$, $\ep >0$  such that 
$$\vert \Delta_{P_0} ^{T+S} ( E(\lambda) , \Phi)^Z- \Delta_{P_0} ^T( E( \lambda), \Phi)^Z \vert  \leq  C e^{ - \ep \Vert T\Vert}( 1+\Vert Z\Vert )^k, $$
for $ \lambda \in \Lambda, Z \in \a_G$, for $T,S$ strictly $P_0$-dominant such that  
$d_{P_0}(T) \geq \delta  \Vert T \Vert $, $\Vert  S \Vert \leq k_0 \Vert T\Vert$, $\Vert T_0 \Vert \leq \Vert  T \Vert $.
\end{lem}

\begin{proof} 
Let us define:
\beq \label{tsq} 
C_{T+S, T}^Q = C^G_{T+S}\cap C_T^Q.
\eeq 
and 
$$ \CC_{T+S, T}^Q= G(F)  N_0( \A) M_0 (C^Q_{T+S,T})K $$
From (\ref{partition}), these subsets of $[G]$ are disjoints. Moreover  from (\ref{partition}), they cover   $ \CC^G_{T+S} $.\\
Let us show that,  for $T,S$ as in the Lemma,  there  exists $\delta_1>0$ such that:
\beq \label{delta1} 
\alpha(X) \geq \delta_1  \Vert X \Vert, X \in  C_{T+S, T}^Q  , \alpha \in \Delta_{0} \setminus  \Delta_{0}^Q
\eeq 
Let $X \in C_{T+S, T}^Q$ and $\Delta_{0} \setminus  \Delta_{0}^Q$.  
The definition  of  $C_T^Q$ shows in particular that $X=  T- X' + Y$ where $X' = \sum_{\beta\in \Delta_{0}^Q} d_\beta  \check{\beta}$ with 
$ d_\beta  > 0$ and $Y \in \a^+_Q$. 
Let $\alpha \in \Delta_{0} \setminus  \Delta_{0}^Q $. 
Since $ \alpha( \check{\beta}) \leq 0$ for each $\beta \in \Delta_{0}^Q$, from the properties of simple roots, one has 
\beq \label{ax}
\alpha(X)  \geq \alpha (T) \geq \delta \Vert T \Vert. 
\eeq 
Let us show: 
\beq \label{ normXct}
\Vert X-T_0 \Vert \leq  \Vert T+S\Vert , X \in C_{T+S}^G .
\eeq  
As $X \in C_{T+S}^G$, $X-T_0 = (T+S)-Y' $ where $Y' $ is a linear combination with coefficients greater  or equal to zero of coroots. 
Moreover $T+S$ and $X-T_0$  are  in $\a_0^+$. Hence 
$$(X-T_0, X-T_0) \leq (X-T_0, T+S) \leq (T+S, T+S)$$
which proves our claim. Hence 
$$\Vert X \Vert  \leq \Vert T \Vert + \Vert S\Vert + \Vert T_0\Vert  \leq (2+k_0) \Vert T\Vert $$ 
and 
$$\alpha(X) \geq \delta_1  \Vert X \Vert $$ 
where $ \delta_1 =  (2+k_0)^{-1} \delta$. This proves (\ref{delta1}).\\ 
From (\ref{ax}) one gets: 
$$ \Vert X \Vert  >> \Vert T \Vert, $$ 
if $Q \not=G$. \\ 
Hence from Lemma \ref{diff}, one gets: 
\ber \label{diff2} 
Let $\Omega$ be a compact subset of $G(\A)$. There exists $ \ep>0$ such that 
$$\vert (\Phi -\Phi^w_Q) ( n_Q exp X\omega ) \vert << e^{ \rho(X)-\ep \Vert T \Vert}, n_Q  \in N_Q( \A),  X \in C_{T+S, T}^Q, \omega \in \Omega,$$ 
and $T, S $ as in the Lemma.
\eer 
Similarly one gets  from Lemma \ref{diffes}: 
\ber  \label{diff2es} 
$$\vert (E_P(  n_Q exp X\omega , \phi, \lambda) -E_P(  n_Q exp X\omega , \phi, \lambda)_Q^w\vert << e^{ \rho(X)-\ep \Vert T \Vert}, $$ 
$$n_Q  \in N_Q( \A),  X \in C_{T+S, T}^Q, \omega \in \Omega, \lambda \in \Lambda.$$
\eer 
Using (\ref{section}), the integration formula linked to the decomposition $G(\A)= N_0( \A)M_0( \A) K$ (cf. \cite{MW}, I.1.13), 
and the fact that $\rho(X)\geq  \rho (T_0)$ for $X\in C_T^G$,  one sees that:  
\ber\label {volumec} 
The volume of $\CC_T^G $ is bounded by a polynomial in $\Vert T\Vert $.
\eer 
Hence: 
\ber \label{volumects} 
The volume of  $\CC^Q_{T+S, T} \subset  \CC_{T+S}^G$ is bounded by a polynomial  in $\Vert T\Vert $,
\eer 
as  $\Vert T+S \Vert  \leq  (1+k_0) \Vert T \Vert$. \\ 
Let us introduce: 
$$I_Q (T, \lambda)
:= \int_{\CC^Q_{T+S, T}} E( x,\lambda)^Z \overline{\Phi}^Z(x) dx,\>  I^w_{Q}(T, \lambda) := \int_{\CC^Q_{T+S, T}} E ( x, \lambda)_Q^{w,Z}  \overline{\Phi}_Q ^{w,Z}(x) dx .$$ 
Notice that: 
\beq \label{igw} 
I_G(T, \lambda)=I_G^w(T, \lambda) = \Omega^T_{P_0}( E(\lambda),  \Phi), \> \sum_{Q\in \cP_{st} } I_Q(T, \lambda) = \Omega_{P_0}^{T+S} ( E(\lambda),  \Phi)
\eeq 
For $C>0, k\in \N$ let us consider the function of $T$ and $Z$: 
\beq \label{et} 
C e^{-\ep \Vert T \Vert} (1+\Vert Z\Vert )^k 
\eeq
It follows from (\ref{diff2}),  (\ref{diff2es}), as well as the tempered estimate for $\Phi$ and the uniform estimate for Eisenstein series (cf. Proposition \ref{uni}) that: 
\ber \label{diffi} 
The difference of 
$$\vert I_Q(T, \lambda)-I_Q^w (T, \lambda) \vert$$ 
is bounded  for $\lambda $, $T$, $S$ as in the Lemma, by  a function of type (\ref{et}). 
\eer 
Let us define 
$$\a^{G,++}_Q(T+S, T):= \{ T_Q +Y\vert Y \in \a^{G,++}_Q, \varpi_\alpha( Y -S) \leq 0, \alpha \in \Delta_Q\}\subset \a^G_Q.$$  
Let us show  
\beq \label{ctsq} 
C^Q_{T+S,T} =  \a_Q^{G,++}(T+S,T) +C^{M_Q}_{T_{M_Q}}.
\eeq 
Let $X\in C_T^{M_Q}$  and $ T_Q + Y \in \a_Q^{G,++}(T+S ,T) $.  
Let  us show that  $X +T_Q +Y$  is an element of $C^Q_{T+S,T} $. 
In view of (\ref{cqt=}), the only thing to prove is that it is an element  $C^G_{T+S,T}$. 
One has: 
$$X+T_Q +Y -S-T=Y-S+X-T_{M_Q}.$$ 
Let $ \alpha \in \Delta_0\setminus \Delta_0^Q$.  
Then $\varpi_\alpha (Y-S+X-T_{M_Q}) = \varpi_\alpha (Y-S)$ which is less than or equal $0$, by the definition of $ \a_Q(T+S ,T) $. Let  $ \alpha \in \Delta_0^Q$. 
The difference $Y-S_Q$ is  a linear combination with coefficients less or equal to zero of elements of $\check{\Delta}_Q$ hence of $\check{\Delta}_0$. 
The same is true for $Y-S= Y-S_Q-S_{M_Q}$ . Hence $\varpi_\alpha(Y-S) \leq 0$. 
The definition of $C_T^{M_Q}$ shows that $\varpi_\alpha ( X-T_{M_Q}) \leq 0$. 
Hence  $\varpi_\alpha (Y-S+X-T_{M_Q}) = \varpi_\alpha (Y-S)\leq 0, \alpha \in \Delta_0^Q$. 
This achieves to prove  $X +T_Q +Y \in C^G_{T+S,T}$ as wanted. Hence 
\beq\label{aqts}
\a_Q^{G,++}(T+S,T)+C_T^{M_Q} \subset C^Q_{T+S,T}.
\eeq 
The reciprocal inclusion follows easily from (\ref{cqt=}) and of the definition of $C^Q_{T+S,T} $. This achieves to prove (\ref{ctsq}).\\ 
We use that $\Phi_Q^w$ and $ E_P( x, \phi, \lambda)_Q^w $ are left $N_Q(\A)$-invariant and that the volume of $N_Q(F) \backslash N_Q( \A)$ is equal to $1$. 
If $P$ is a parabolic subgroup of $G$ with Levi subgroup $M_P$ and $k \in K$, we have defined(cf. section \ref{transi}) for any function on $G(\A)$, 
the function $\phi^{M_P,k}$ on $M_P( \A)$ by: 
$$ \phi^{M_P,k}  ( m) = e^{-\rho(H_P(m))}\phi(mk), m\in M_P(\A).$$ 
Thus, using (\ref{intf}), we get: 
$$I_Q^w(T, \lambda) = \int_{\CC_{T+S, T}^Q } E( x,  \lambda)_Q^{w,Z}  \overline{\Phi}_Q^{w,Z} (x) dx  =$$ 
$$ \int_{ \a^{G,++}_Q( T+S, T)\times \CC_T^{M_Q} }E( exp X m^1_Q k, \lambda )_Q^{w,Z} \overline{\Phi}_Q^{w,Z} (exp X m^1_Qk)e^{-2 \rho_Q(X)}dX dm^1_Q dk $$ 
$$ =\int_{ \a^{G,++}_Q( T+S, T)\times \CC_T^{M_Q}    \times K} \Omega^T_{P_0 \cap M_Q } ([E(\lambda)_Q^w]^{M_Q,k},  [ \Phi_Q^w]^{M_Q,k} )^{X+Z}dm^1_QdXdk.$$ 
Recall that by induction hypothesis, the Theorem \ref{tr} is true for $M_Q$ if $Q\neq G$ and one can apply this induction hypothesis. 
Taking into account (\ref{volumects}) and the previous equality, one sees, using $K$-finiteness, that the  difference of the preceding expression 
with the same expression, where $\Omega^T_{P_0 \cap M_Q} $ is replaced by $\omega^T_{P_0 \cap M_Q} $, denoted $J_Q(T, \lambda) $, 
is bounded by a function of type (\ref{et}).\\ 
One has: 
$$J_Q (T, \lambda) = \int_{ \a_Q^{G,++} (T+S,T)\times K}   \omega^T_{P_0 \cap M_Q } ([ E(\lambda)^w_Q ]^{M_Q,k},  [ \Phi_Q^w]^{M_Q,k} )^{X+Z} dXdk $$ 
$$\>\> =   \int_{ \a_Q^{G,++}  (T+S,T)\times K}  \sum_{ R_1\in  \cP_{st}(M_Q) } r^T_{R_1} ( [E(\lambda)^w_Q]^{M_Q,k},  [ \Phi_Q^w]^{M_Q,k} )^{X+Z} dXdk.$$ 
If $R_1$ is a standard parabolic subgroup of $M_Q$, let $P_1$ be the standard parabolic subgroup of $G$ contained in $Q$ with $P_1 \cap M_Q=R_1$. 
Using  Lemma \ref{transw}  (iv),  the definition  (\ref{rtp}), for $M_Q$ and $R_1$,  and integrating over $K$,  one sees: 
$$J_Q (T, \lambda)
= \sum_{P_1 \in \cP_{st}(G),  P_1 \subset Q}\int^*_{\a_Q^{G,++} (T+S,T) +\a_{P_1 \cap M_Q,-}^{M_Q}(T)  } ( E(\lambda) _{P_1}^w ,   \Phi_{P_1} ^w )^{X+Z} dX .$$ 
We observe that $J_G(T, \lambda) = \omega^{T}_{P_0} (E(\lambda), \Phi)$ and one has seen that $I_G (T, \lambda)= \Omega^{T}_{P_0} (E(\lambda), \Phi)$. 
One writes: 
$$ \Omega^T_{P_0}  ( E(\lambda), \Phi)= \Delta^T_{P_0}  ( E(\lambda), \Phi)+ \omega ^T_{P_0}  ( E(\lambda), \Phi).$$ 
Using what we have just proved and (\ref{igw}), and (\ref{diffi}), we get: 
\ber \label{modu}
The modulus of the difference 
$$ \Omega ^{T+S}_{P_0} ( E(\lambda), \Phi) -   \Delta^T_{P_0}  ( E(\lambda), \Phi)$$ 
$$= \Delta^{T+S}_{P_0} ( E(\lambda), \Phi) +\omega^{T+S}_{P_0} ( E(\lambda), \Phi)-  \Delta^{T}_{P_0} ( E(\lambda), \Phi)$$ 
with $J(T, \lambda)= \sum_{Q \in \cP_{st} (G) } J_Q(T, \lambda)$ is bounded by a function of type (\ref{et}). 
\eer 
Thus it is enough, to finish the proof of the Lemma, to prove: 
$$J (T, \lambda)= \omega^{T+S}_{P_0} ( E(\lambda), \Phi).$$ 
Using the expression of $J_Q (T, \lambda)$ above and interverting the sum over $Q$ and $P_1$, one sees that: 
$$J (T, \lambda)
= \sum_{ P_1\in \cP_{st}, G , Q\in \cP_{st} (G), P_1 \subset Q} 
\int^*_{\a_Q^{G,++} (T+S,T)+\a_{P_1 \cap M_Q,-}^{M_Q}(T)+Z } e^{-2 \rho_{P_1}(X)}(E(\lambda)_{P_1}^w ,  \Phi_{P_1} )^{X+Z}  dX.$$  
Let $ \a^{M_Q} _{P_1\cap M_Q, --} $  be the interior in $\a^Q$ of $ \a^{M_Q} _{P_1\cap M_Q, -} $ and let $ \a_Q^{G,+}(T+S,T) $ be the closure of 
$\a_Q^{G,++}(T+S,T) $ in $\a_Q$. 
Let us show:
\ber \label{partition2} 
The union 
$$\cup_{Q\in \cP_{st} ,P_1 \subset Q} \a_Q^{G,+}(T+S, T)  +\a_{P_1 \cap M_Q,--}^{M_Q}(T_{M_Q} )$$ 
is disjoint and is a partition of  $ \a_{P_1, -}^G (T+S)$.
\eer
Let us consider the projection of $\a_{P_1} $ on  the closed convex cone $\a_{P_1 ,-}^{G}$. 
By translating, one sees, using e.g. \cite{C} Corollary 1.4, that, if $X\in \a_{P_1 ,-}^G(T+S)$, there exists a unique  standard parabolic subgroup of $G$, 
$Q$ with $P_1 \subset Q$ such that $X= X' +Y$, $X' \in \a_{P_1\cap M_Q,--}^{M_Q}(T_{M_Q}), Y \in \a^{G, +}_Q (T)$. As $X \in  \a_{P_1 ,-}^G(T+S) $, 
one has $Y \in \a_Q^{G,+}(T+S, T)$. 
Hence the union in (\ref{partition2}) contains $ \a_{P_1, -}^G (T+S)$ and is disjoint.\\ 
Reciprocally let us prove that for $P_1\subset Q$: 
$$ \a_{P_1\cap M_Q,--}^{M_Q}(T_{M_Q})+\a^{G, +}_Q (T+S,T)\subset \a_{P_1, -}^G (T+S).$$ 
To see this, by translation , it is enough to  prove that if $X\in \a_{P_1\cap M_Q,--}^{M_Q}(T_{M_Q}),Y \in \a^G_{Q, -}$ one has 
$X +Y \in \a_{P_1,-}^G $  which is clear by convexity. This proves (\ref{partition2}).\\ 
Neglecting sets of measure zero, this implies that  the sum $J(T, \lambda) $ is equal to $\omega^{T+S}_{P_0} (E(\lambda), \Phi)^Z$. 
This achieves to prove the Lemma.
\end{proof} 

We will give below a proof  Theorem \ref{tr}.  
It is done  using first  the argument of \cite{Alt}, Lemma 9.2 and second using wave packets as in \cite{D2} Lemma 3 and end of proof of Proposition 1 
(see also the  end of the proof of Theorem 1 in \cite{D1}).\\ 
One fixes $\delta>0$ and one  writes $lim_{T \xrightarrow{\delta} \infty }$ to describe the limit when $\Vert T \Vert $ tends to infinity  verifying $d(T) \geq \delta \Vert T \Vert$. 
One deduces from the preceding Lemma, as in ( \cite{Alt}, Lemma 9.2) that the limit 
$$\Delta_{P_0}^\infty(E( \lambda), \Phi)^Z = lim_{T\xrightarrow{\delta} \infty} \Delta_{P_0}^T(E( \lambda), \Phi)^Z$$ 
exists uniformly for  $\lambda$  in any  compact subset  of $\cH^c$ and if $\Lambda $ is a bounded set  in $\cH^c$, 
there exists $C, \ep>0$, $k \in \N$ such that for $ \lambda\in \Lambda$ and  and $T$  such that $d(T)\geq \delta \Vert T \Vert $, $Z \in \a_G$, one has: 
\beq \label{deltat} 
\vert \Delta_{P_0}^\infty(E( \lambda), \Phi)^Z -\Delta_{P_0}^T(E( \lambda), \Phi)^Z\vert \leq C e^{-\ep \Vert T \Vert }( 1 +\Vert Z\Vert)^k, \lambda \in \Lambda, Z \in \a_G.
\eeq
We prepare some Lemmas to prove that $\Delta_{P_0}^\infty(E( \lambda), \Phi)^Z $ is identically zero on $\cH^c$. 
Using Proposition \ref{wps}, we  define a distribution $T_{\Phi,Z} $ on $ i \a_P^*$ by: 
$$T_{\Phi,Z} ( a)= \int_{ [G]^1} E_a (x)^Z \overline{\Phi}^Z(x) dx, a \in C_c^\infty( i\a_P^*), $$ 
where $E_a$ is the wave packet $\int_{i\a_P^*} a(\lambda) E(\lambda) d\lambda.$

\begin{lem} 
The support $\cS$ of $T_{\Phi,Z} $ is  a finite set.
\end{lem}

\begin{proof}
For $ \lambda\in i\a_P^*$, the center $Z( \g_\infty)$ of $U( \g_\infty)$ acts on $E(\lambda)$ by a character denoted 
$\chi_\lambda$ and $\Phi$ is annihilated by an ideal $I$ of $Z(\g_\infty)$ of finite codimension. Let us compute  in two ways: 
$$ A:= \int_{[G]^1} ( z E_a(x))^Z  \overline{\Phi}^Z(x) dx,z \in Z([\g_\infty, \g_\infty] )\subset Z( \g_\infty), z^*\in I,$$
where $z^*$ is the adjoint of $z$.\\ 
On one hand, looking to the action of $z $ on $E(\lambda) $ and differentiating under the integral defining $E_a$ we get: 
$$A= T_{\Phi ,Z}( p(z) a),$$
where $p(z)(\lambda)= \chi_{\lambda}(z)$, which is a polynomial in $\lambda$. 
On the other hand: 
$$A= \int_{[G]^1} (  E_a(x))^Z  \overline{z^*\Phi}^Z(x) dx=0$$ 
From the equality above,  if $ z^*\in I$ the distribution $p(z) T_{\Phi, Z} $ is equal to zero. Let $I^*= \{ z^*\vert z \in I\}$. 
As $I$ is finite codimensional, the set of $\lambda \in i\a_P^*$ such that $I^* \subset \ker \chi_\lambda$ is a finite set $\cF$. 
Hence if $\lambda\notin \cF$, there exists $z \in I^*$ such that $p(z) (\lambda) \not= 0$. 
Hence $T_{\Phi,Z}$ restricted to a neighborhood of $\lambda$ is zero. Hence $\cS\subset \cF$. 
\end{proof}

\begin{lem} 
If $ a \in C_c^\infty ( i \a_P^*)$  has its support in the complimentary set of $\cS$, one has:
$$ \lim_{T \xrightarrow{\delta} \infty }\int_{i \a_P^*} a(\lambda) \Omega^T_{P_0}(E(\lambda), \Phi)^Zd\lambda =0$$
\end{lem}

\begin{proof}
From Fubini theorem  and Lebesgue dominated convergence the limit is equal to $T_{\Phi,Z} (a)$, which is equal to zero by the preceding lemma. 
\end{proof}

\begin{lem}
If $ a \in C_c^\infty ( i \a_P^*)$  has its support in $\cH^c$ one has:
$$ \lim_{T \xrightarrow{\delta} \infty }\int_{i \a_P^*} a(\lambda) \omega^T_{P_0}(E(\lambda), \Phi)^Zd\lambda=0.$$
\end{lem}

\begin{proof}
This follows from the definition of $\omega^T_{P_0}(E(\lambda), \Phi)^Z$, (\ref{rtp=}) and from the fact that the  Fourier transform of 
a $C_c^\infty$ function on $\R^n$ is rapidly decreasing.
\end{proof} 

\begin{lem} 
If $\cS^c$ is the complimentary set of $\cS$ in $ i\a_P^*$  one has:
$$\Delta^\infty _{P_0}(E(\lambda), \Phi)^Z =0, \lambda \in \cH^c \cap \cS^c.$$
\end{lem}

\begin{proof}
From the two preceding lemmas one has for all in $C_c^\infty(i\a_P^*)$ with support in the intersection   $\cH^c \cap \cS^c$:
$$ \int_{i\a_P^*} a(\lambda) \Delta^\infty_{P_0}(E(\lambda), \Phi)^Zd\lambda=0 .$$
This implies   the Lemma.
\end{proof}

\begin{proof}
Let us finish the proof of the Theorem \ref{tr}. 
The vanishing property of the preceding Lemma together with (\ref{deltat}) shows that the bound  of the theorem is true for $\lambda$ in 
a bounded subset of $ \cH^c\cap \cS^c$. Recall that $ \Omega^T_{P_0}(E(\lambda), \Phi)^Z $ is analytic in $\lambda$. 
Hence, for any $\lambda$ in $i\a_P^*$ and any compact neighborhood of $\lambda$, $V$, $ \omega^T_{P_0}(E(\lambda), \Phi)^Z$ is bounded on 
$V\cap  \cH^c\cap \cS^c$. 
But this meromorphic function has only possible poles along hyperplanes. 
It follows that it is analytic on $i\a_P^*$. This proves the first part of the Theorem. The second part  follows from (\ref{deltat}) by continuity and density.
\end{proof}

Let  $\Lambda\in \a_{M_P}^*$ be strictly $P$-dominant. 
If $ Q$ is a parabolic subgroup of $G$ with Levi subgroup $M_Q$ one  will denote by $ \psi^\Lambda_Q$ the characteristic function of :
$$C^\Lambda_Q  = \{ X \in  \a_P^G \vert \omega_ \alpha(X)\Lambda( \check{\alpha}) >0, \alpha \in \Delta_Q\}$$
that we look as  a tempered measure  on $\a_P^G$ by our choice of Haar measures. 
Let $\beta_Q^\Lambda $ be the number  of elements  $ \check{\alpha} $ of  $\check{\Delta}_Q $ such that $\Lambda( \check{\alpha} ) <0$.
Then one has the following  proposition, whose proof  is analogous to Proposition 2 in \cite{D2}, using  (\ref{eq}) and (\ref{rtp=}). 

\begin{prop}\label{cone} 
Using the notation of Theorem \ref{tr},  the analytic function $\lambda \to \omega^T_{P_0} (E(\lambda, \Phi)^Z $ is equal, as a distribution on $i\a_P^*$, to the sum:\\
(a) on $Q\in \cP_{st}$ \\
(b) on $s \in W( Q\vert P)$\\
(c) on $\mu \in \cE_Q(\Phi_Q^w)$ of:
$$C_Q [(p_Q(M(s^{-1}, \lambda) \phi, \Phi^w_{Q, 0, \mu}  )^{Z} \circ s^{-1} )( \partial) (-1)^{\beta^\Lambda _{Q^s}} (\psi^\Lambda_{Q^s,  T_{Q^s}+Z })^{\hat{}}\>]( \lambda
-s^{-1} \mu),$$
where $Q^s= sQs^{-1}$ and $ \psi^\Lambda_{Q^s,  T_{Q^s+Z}} $ is the characteristic function of the translate of $C^\Lambda_{Q^s} $, $C^\Lambda_{Q^s} -T_{Q^s} -Z$ 
and the upper index $\hat{}$ indicates that we take its Fourier transform. 
\end{prop}

\begin{proof}
First $ \omega(\lambda):= \omega^T_{P_0} ( E( \lambda), \Phi)$ is analytic on   $\i\a_P^*$  from Theorem \ref{tr} (i). 
Then, $\omega( \lambda)$ is the limit when $t$  to $0^+$ of $\omega( \lambda+t\Lambda)$ in the sense of distributions. 
Then one uses \cite{D2}, Lemma 11, for each term of the sum defining $\omega(\lambda + t\Lambda)$. The Lemma follows. 
\end{proof} 


The following theorem is the main result of this article. 

\begin{theo} 
The image of the map $\overline{\cE}$ of Theorem \ref{lapid} is equal to $L^2(G(F) \backslash G( \A))$.  
\end{theo}

We start with a preliminary remark. 
Let us recall  some result of Bernstein in \cite{B}. We use the notation and terminology of l.c. 
The space $[G]$ is  of polynomial growth whose rank $rk([G])$ is equal to the split rank of $G$ (cf. l.c., section 4.4,  Example 1).  
Using the definition of  a space of polynomial growth in l.c. p.689 and what follows this definition, one gets that for $d > rk([G])$, $(1+ \sigma_{[G]})^d$ is a summable weight. 
Moreover the  automorphic forms which  may contribute  to the spectrum  (see below for a precise meaning) are in some 
$L^2( X, (1+\sigma_{[G]})^{-d} dx)$, for some $d>0$, as well as their derivatives. Hence they are tempered, by Proposition \ref{key}.\\

\begin{lem}
If the image of $\overline{\cE}$ is not equal to $L^2(G(F) \backslash G(\A))$, there  exists   a non zero tempered  automorphic form $\Phi$, 
transforming under  a unitary character $\nu_G  $  of $\a_G$ and orthogonal to all the wave packets $E_a$  of Proposition \ref{wps},  
when $P$, $\phi $ and $a$ varies. 
\end{lem}

\begin{proof} 
The proof is similar to \cite{CD}, Lemma 11. 
Let $\cH$ be the orthogonal to the image of $\cE$,  which is assumed to be non zero. 
One considers the decomposition of this representation of $G(\A)$ into an Hilbert integral of multiple of irreducible representations: 
$$\cH= \int_{\hat{G}}\cH_\pi d\mu( \pi).$$ 
The restriction $\xi$ of the Dirac measure at  the neutral element  to   the space $\cH^\infty$ of $C^\infty$ vectors, desintegrates: 
$$\xi= \int_{\hat{G}} \xi_\pi  d\mu( \pi), $$ 
where $\xi_\pi  \in (\cH_\pi ^{-\infty})^{G(F)}$, i.e. is a $G(F)$-invariant distribution vector on $\cH_\pi$.\\ 
Let  
$$v=\int_{\hat{G}}v_\pi d \mu(\pi) \in \cH^\infty.$$ 
We assume that it is $K$-finite and  non zero.\\
Let  $(g_n) $    be  a dense sequence in $G(\A)$. For $p,q, n \in \N$, let: 
$$X_{p,q,n} = \{\pi \in \hat{G}\vert  \> \vert < \xi_\pi, \pi(g_n)v_\pi > \vert  \leq p\Xi(g_n) ( 1+\sigma\gg (g_n))^{q} \} $$ 
For all $g \in G(\A)$, the map $\pi \mapsto < \xi_\pi ,  \pi(g) v_\pi> $ is $\mu$-measurable. 
Hence all $X_{p,q,n}$ are measurable as well as $X_{p,q} = \cap_{n \in \N}X_{p,q,n}$. 
Moreover, from our preliminary remark, just after the statement of the Theorem, $\cup_{p,q \in \N}X_{p,q} $ is equal to $\hat{G}$ up to a set of $\mu$-measure 0. 
Let $X_{p,q}^0$ be the  set  of elements $\pi$ of $X_{p,q}$ of $\pi$ such that $g \mapsto < \xi_\pi , \pi(g) v_\pi>$ is non identically zero.  
As $v$ is non zero, one can find $p,q$ such that the set  $X_{p,q}^0$ is of non zero measure. 
Let $\chi$ be the characteristic function of $X^0_{p,q}$. Then one has for any $ \theta \in L^{\infty} (\hat{G}, \mu)$, going back to the definition:
$$ f_\theta: = \int_{ \hat{G}}  \chi(\pi) \theta(\pi) v_\pi d\mu(\pi) \in \cH^\infty.$$ 
Hence by using the desintegration of $\xi$, viewing $f_\theta$ as a function on $G(\A)$,  one has: 
$$f_\theta (g) = \int_{ \hat{G}} \chi( \pi) \theta  ( \pi) < \xi_\pi , \pi(g) v_\pi>  d \mu(\pi), g \in G(\A). $$ 
Let us show that the map $(\pi, g) \mapsto  < \xi_\pi, \pi(g) v_\pi> $ is measurable. 
Let $g = g_\infty g_f$ where $g \in G( \A_\infty)$ and $g_f \in G(\A_f)$. As $v $ is smooth the map is locally constant in $g_f$. 
One easily reduces to $g_f=1$ and look to the dependence on $( \pi,g_\infty)$ only. Then one uses the argument given in \cite{CD}, p. 96 which uses step functions.\\ 
Using  (\ref{sixi}), one can apply  Fubini's theorem to  
$$ \int_{[G]} E_a(x)f_\theta(x) dx = \int_{X^0_{p, q}} \theta(\pi) \int_{[G] }E_a(x)  \overline{< \xi_\pi , \pi(x) v_\pi>}  dx d \mu( \pi). $$ 
This has to be zero for all $\theta$. Hence for almost all $\pi$ in $X^0_{p,q}$ one has:
$$ \int_{[G] } E_a (x)  < \xi, \pi(x) v_\pi> dx= 0$$ 
for a given $E_a$. Using  a separability argument, one finds an element  $\pi_0$ of $X^0_{p, q} $ such that it is true for all $E_a$. One takes $\Phi= < \xi, \pi_0(x) v_\pi>$. 
\end{proof} 
Let  $ a = a_1 \otimes a_2$ where $a_1  \in C_c^\infty( (i\a_P^G)^*)$ and  $a_2  \in C_c^\infty( i\a^*_G)$ ). Let $\nu_G\in i\a_G^* $ which describes 
the  action of $\a_G$ on $\pi_0$. Then, using Fourier inversion formula for $i\a_G^*$, one has: 
$$\int_{ G(F) \backslash G( \A)}   E_a (x) \overline{ \Phi(x)} dx = a_2(\nu_G ) \int_{ G(F) \backslash G( \A)^1} E_{a_1}  (x)  \overline{ \Phi(x)} dx $$ 
where $E_{a_1}= \int_{i\a_P^{G,* } }a_1( \lambda )  E_P(x, \phi, \lambda) d \lambda$.
We want to compute $$ I= \int_{ G(F) \backslash G( \A)^1} E_{a_1}  (x)   \overline{ \Phi(x)} dx $$ using the preceding theorem. 

\begin{lem} 
$$I= C_P  \sum_{  \mu \in \cE^w_P(\Phi)  } [(p_P (\phi  , \Phi_{P, 0, \mu} )^{0})  ( \partial)  a_1 )](\mu_{\vert \a^G_P}).$$
\end{lem} 

\begin{proof} 
We  can compute $I$ as limit.  Using Lebesgue dominated convergence and Fubini theorems, one can write $I$ as a limit. 
Let $T $ be strictly $P_0$-dominant. Then:
$$I= lim_{n \to + \infty} \int_{i\a^{G, *}_P } a_1( \lambda ) \Omega^{nT}_{P_0} ( E(\lambda),  \Phi)^0 d\lambda.$$
From Theorem \ref{tr}, one can replace $\Omega$ by $ \omega$.\\ Then one uses Proposition \ref{cone}.  
One sees easily that unless $Q^s= P$, the characteristic function of $C^\Lambda_{Q^s} - nT_{Q^s}$ tends to $0$ in the sense of tempered distributions. 
But  in this case $Q$ is standard and conjugate to $P$. Hence $Q=P$ and $s =1$.  Using Proposition \ref{cone}, one computes easily the limit.
\end{proof}

Now we can finish the proof of the Theorem. The hypothesis on $\Phi$  above shows that the right hand side of the equality of the Lemma  is zero 
for all $P$,  $\phi$,  $a_1, a_2$.   
One concludes, by varying $a_1, a_2$ and $\phi$, that  $\Phi_{P, 0, \mu}  \in \cA_P^{temp,c} (G)$  for all $P$ and $\mu \in \cE^w_P(\Phi) $. 
Then, using Lemma \ref{tempexp} (ii),  one concludes that $\Phi= 0$.  A contradiction which finishes the proof.  
\qed

\noindent Patrick Delorme 
\\patrick.delorme@univ-amu.fr \\
Aix Marseille Univ, CNRS, Centrale Marseille, I2M, Marseille, France

\end{document}